\documentclass[10pt]{amsart}
\usepackage[utf8]{inputenc}
\usepackage{comment}
\usepackage{amscd}
\usepackage{todonotes}
\usepackage{amssymb}
\usepackage{empheq}
\usepackage{pstricks}
\usepackage{graphicx}
\usepackage{lineno}
\usepackage{subcaption}
\usepackage{appendix}
\usepackage{xcolor}
\usepackage{textcomp}
\usepackage{tikz-cd} 

\newtheorem{definition}{Definition}
\newtheorem{theorem}{Theorem}

\newtheorem{lemma}{Lemma}
\newtheorem{corollary}{Corollary}
\newtheorem{proposition}{Proposition}
\newtheorem{remark}{Remark}
\newtheorem{example}{Example}

\numberwithin{equation}{section}
\newtheorem{Remark}{Remark}

\title[]{Generalized Evolution Semigroups and $h-$Dichotomies for Evolution Families on Banach Spaces}

\author[]{\'Alvaro Casta\~neda}
\author[]{Ver\'onica Poblete}
\author[]{Gonzalo Robledo}

\address{Universidad de Chile, Departamento de Matem\'aticas. Casilla 653, Santiago, Chile}

\email{castaneda@uchile.cl, vpoblete@uchile.cl,  grobledo@uchile.cl }
\thanks{The first and third authors were funded by FONDECYT Regular Grant 1240361.}

\begin{document}

\begin{abstract}
    This paper develops a comprehensive theory generalizing exponential decay patterns for evolution processes in Banach spaces. We replace classical exponential bounds with more flexible decay rates governed by an increasing homeomorphism $h$. The core of our approach lies in constructing particular group structures induced by $h$, which allow us to define generalized semigroups on function spaces. We prove that these $h$-semigroups are equivalent to classical evolution semigroups through a natural transformation. Our main result establishes that three fundamental concepts are equivalent: hyperbolicity of the generalized semigroup, dichotomy of the underlying evolution process, and a spectral condition on the generator. This work extends classical dichotomy theory to encompass a wider class of decay patterns, providing new tools for analyzing asymptotic behavior in dynamical systems.
\end{abstract}

\maketitle

\section{Introduction}

\subsection{Preliminaries} The property of exponential dichotomy  
for evolution families $\{U(t,s)\}_{t\geq s}$ on a Banach space $(X,||\cdot||_{X})$ has been widely studied in the literature \cite{Henry-0,Henry,LRS,LM, Pazy,  VQP, SS}. A noteworthy consequence is the existence of an associated $C_{0}$--semigroup $\{T_{t}\}_{t\geq 0}$ together with its corresponding infinitesimal generator $G$. A classical result \cite[Th. 3.17 and Th. 4.25]{CL} and \cite[Th. VI.9.18]{EN} establishes an equivalence
between three distinguished properties:

\begin{equation}
\label{EQ}
\left\{\begin{array}{ll}
\bullet & \textnormal{The exponential dichotomy of the family $\{U(t,s)\}_{t\geq s}$},\\ 
\bullet & \textnormal{the hyperbolicity of $\{T_{t}\}_{t\geq 0}$},\\
\bullet & \textnormal{a geometrical property of the resolvent $G$, namely},\\
& \textnormal{its intersection with the imaginary axis is empty}.
\end{array}\right.
\end{equation}

An essential assumption of this equivalence is that the evolution family is \textit{exponentially bounded}, that is there exist $K\geq 1$ and $\alpha>0$ such that
\begin{equation}
\label{decroissance}
||U(t,s)||\leq Ke^{\alpha(t-s)} \quad \textnormal{for any $t,s\in \mathbb{R}\times \mathbb{R}$ with $t\geq s$}.
\end{equation}

The main goal of this article is to study the above equivalences --in a generalized framework-- under the existence of bounds encompassing (\ref{decroissance}), namely,
\begin{equation}
\label{decroissance-bis}
||U(t,s)||\leq K\left(\frac{h(t)}{h(s)}\right)^{\alpha}
\quad  \textnormal{for any $t,s\in \mathbb{R}\times \mathbb{R}$ with $t\geq s$},
\end{equation}
where $h\colon \mathbb{R}\to \mathbb{R}^{+}$ is a strictly increasing homeomorphism generalizing the exponential function. The idea behind to consider these maps $h(\cdot)$ is to allow the rates of contraction and expansion to be determined by general growth functions describing the behavior of solutions of differential equations and has been developed in a systematic way in the works of Martin \cite{Martin}, Muldowney~\cite{Muldowney} and Naulin–Pinto~\cite{NaulinPinto}.

In order to relate the bounds (\ref{decroissance}) and (\ref{decroissance-bis}) we will recall
the tough--provoking article of Pe\~na and Rivera--Villagr\'an \cite{JFP}, where the authors noticed -- in a slighlty different context-- that the
change of time variables
\begin{equation}
\label{chv}
t=h^{-1}(e^{\tilde{t}}) \quad \textnormal{and} \quad \tilde{t}=\ln(h(t)),
\end{equation}
transforms (\ref{decroissance-bis}) as follows:
\begin{equation}
\label{decroissance-3}
||U(h^{-1}(e^{\tilde{t}}),h^{-1}(e^{\tilde{s}}))||\leq K e^{\alpha(\tilde{t}-\tilde{s})}.
\end{equation}

Furthermore, a noteworthy remark of \cite{JFP} was to explain that the above transformation (\ref{chv}) does not trivially transforms the general bound (\ref{decroissance-bis}) into an exponential one. In fact the change of variables reveals something much more subtle, namely, a commutative diagrams of isomorphisms of topological groups
such that
\[
\begin{tikzcd}
(\mathbb{R},+) \arrow{r}{e}  & (\mathbb{R}^{+},\cdot) \arrow{d}{h^{-1}} \\
& (\mathbb{R},*_{_h})
\end{tikzcd}
\quad \textnormal{and} \quad
\begin{tikzcd}
(\mathbb{R},*_{_h}) \arrow{r}{h} & (\mathbb{R}^{+},\cdot) \arrow{d}{\ln} \\
& (\mathbb{R},+),
\end{tikzcd}
\]
where the topological abelian group $(\mathbb{R},*_{_h})$ -- a formal definition will be given later -- is such that the isomorphism $h\colon  (\mathbb{R},\ast_{h})\to (\mathbb{R}^+, \cdot)$ verifies $h(t\ast s)=h(t)h(s)$,
which mimics the exponential identity $e^{t+s}=e^te^s$. Moreover, as an abelian group is a $\mathbb{Z}$--module,
we note that its external composition law can be written in terms of $*_{h}$ and will be denoted by $\odot$.

Motivated by \eqref{decroissance-3}, it is natural to ask what happens when the asymptotic behaviour of
an evolution family is not governed by a fixed exponential rate, but rather by a more
general growth rate. In particular, we would like to replace the classical weight
$e^{(t-s)}$ by ratios of the form $h(t)/h(s)$, and to study the resulting
notions of stability and dichotomy in a semigroup framework adapted to this new time
scale.

To this end, we fix a map $h(\cdot)$ and endow $\mathbb{R}$ with the induced
operations $*_{_h}$ and $\odot$, obtaining a one-dimensional normed vector space
$(\mathbb{R},*_{_h},\odot)$ with neutral element $e_*$. 
Families
$U = \{U(t,s)\}_{t\ge s}$ are then indexed by $(\mathbb{R}_*,*_{_h},e_*)$. The family $U$ is called $h-$ evolution family,  and its growth and decay 
are measured with respect to the $h$–time rather than the classical
linear time. This leads naturally to the concept of \emph{$h$-dichotomy}, where the
role of the exponential weights is played by the ratio $h(t)/h(s)$.

In this context, we can observe a growing interest in dichotomies and decays beyond the exponential case, namely, the work
of Megan, Sasu and Sasu \cite{MSS} devoted to nonuniform exponential dichotomy and the works of Lupa and Popescu \cite{Lupa2,Lupa} focused in a generalized exponential dichotomy. With respect to nonexponential decays we refer to the seminal and tough provoking work of Borichev and Tomilov \cite{BT} which provided conditions ensuring a polynomial decay.

\medskip

\subsection{Structure and Novelties of the article} 

\medskip

The present paper is organized as follows. In 
Section 2 we revisit the abelian group $(\mathbb{R},*_{_h})$ which is totally ordered allowing
to define an absolute value $|\cdot |_*$. Moreover, as a novelty, we prove that the vector space $(\mathbb{R},*_{_h},\odot)$ is complete,
denoting it as $\mathbb{R}_* = (\mathbb{R},*_{_h},\odot, | \cdot |_*)$.

In Section~3 we regard $C(\mathbb{R}_*,X)$ as a normed vector space and introduce the notion
of an $h$-semigroup acting on $C(\mathbb{R}_*,X)$. Given an $h$-evolution family $U$ on $X$,
we associate to it an operator family $\{T_t\}_{t\ge e_*}$, indexed by $(\mathbb{R}_*,*_{h})$, via
\[
(T_tu)(s) := U\bigl(s,\,s*_{h}t^{*-1}\bigr)\,u\bigl(s*_{h}t^{*-1}\bigr),
\qquad u\in C(\mathbb{R}_*,X),\ s\in\mathbb{R}_*,
\]
we then show (cf. Lemma \ref{C0semigroup}) that, if $U$ satisfies \eqref{decroissance-bis}, the family
$\{T_t\}_{t\ge e_*}$ forms a so called $h$-semigroup on $C(\mathbb{R}_*,X)$.

A first step to generalize the equivalences (\ref{EQ}) is provided by Theorem~\ref{AgenB}, which gives a precise
characterization of the generator of the $h$-semigroup $\{T_t\}_{t\ge e_*}$. As a
consequence, we obtain detailed spectral information for this $h-$semigroup. The proof
relies on the construction of a classical $C_0$-semigroup associated with the original
$h$-evolution family $U$, and on a conjugacy map that transfers spectral properties
between the classical and the $h$-reparametrized settings.

In Section~4, as a second step to address the generalization of the equivalences (\ref{EQ}), we develop the semigroup approach further by relating the notion of
$h$-dichotomy to the spectral theory of the associated evolution $h$-semigroup on
$C_0(\mathbb{R}_*,X)$. In particular, we show that $h$-dichotomy of $U$ is equivalent
to the invertibility of the generator $B_h$ of $\{T_t\}_{t\ge e_*}$ together with the
spectral gap condition
\[
\sigma(B_h)\cap i\mathbb{R}=\varnothing.
\]
This yields an operator-theoretic characterization of $h$-dichotomy which extends, to
the $h$-framework, the classical results for evolution families on the half-line; see \ Theorem~\ref{h-dichotomy}
which states the equivalence between the following properties
\begin{equation*}
\left\{\begin{array}{ll}
\bullet & \textnormal{The $h$--dichotomy of the family $\{U(t,s)\}_{t\geq s}$},\\ 
\bullet & \textnormal{the hyperbolicity of the $h$--semigroup $\{T_{t}\}_{t\geq e_{*}}$},\\
\bullet & \textnormal{the geometrical property $\sigma(B_h)\cap i\mathbb{R}=\varnothing$}.\\
\end{array}\right.
\end{equation*}

\section{Algebraic Preliminaries and contextualization}
We begin by recalling the notion of a growth rate, and then introduce the algebraic
and topological structures it induces on the real line. These constructions provide
the basic framework and tools that will support the rest of the paper.

\begin{definition}
    We will say that the function $h: \mathbb{R} \to \mathbb{R}^+$ is a growth rate if $h$ is a strictly increasing homeomorphism.  
\end{definition}

Some examples of growth rates are given by $h(t)=e^{t},
    \ h(t)=e^{(t-t_{0})^{n}}$ for any odd $n\in \mathbb{Z}^{+}$ and
$h(t)=t+\sqrt{t^2+1}$.

\medskip
The growth rates allows to define the following laws of composition on $\mathbb{R}$:
\begin{equation}
\label{LCI}
\begin{array}{rcl}
\mathbb{R} \times \mathbb{R} & \to &  \mathbb{R}  \\
(t,s) &\mapsto &  t\ast_{_h} s:=h^{-1}\left(h(t)h(s)\right),
\end{array} 
\end{equation}
and
\begin{equation}
\label{LCE}
\begin{array}{rcl}
\mathbb{R} \times \mathbb{R} & \to &  \mathbb{R}  \\
(\alpha,t) &\mapsto &  \alpha \odot t:=h^{-1}\left(h(t)^{\alpha}\right),
\end{array} 
\end{equation}
which are well defined since $h(t)^\alpha>0$ for every $\alpha\in\mathbb R$.

\begin{remark}
Notice that, when considering $h(t)=e^{t}$, the above composition laws becomes
$(t,s)\mapsto t+s$ and $(\alpha,t)\to \alpha t$, namely, the classical addition and
the scalar multiplication in $(\mathbb{R},+,\cdot)$.
\end{remark}


\subsection{Abelian ordered group induced by the growth rate $h$}
The composition law (\ref{LCI}) allows to state the following result, which has been stated without proof in \cite[Proposition 1]{EPR} and \cite[Sect.3]{JFP}: 
\begin{proposition}
\label{Groupe}
The pair $(\mathbb{R},\ast_{_h})$ is an abelian group where the unit element and the inverse for any $t \in \mathbb{R}$ are respectively defined by: 
\begin{equation}
\label{inverse}
e_{*}:=h^{-1}(1) \quad \textnormal{and} \quad
t^{*-1}:= h^{-1}\left( \frac{1}{h(t)}\right).
\end{equation}
\end{proposition}

\begin{proof}
Notice that the operation (\ref{LCI}) is commutative
\begin{displaymath}
\begin{array}{rcl}
t \ast_{_h}s =  h^{-1}(h(t)h(s))
= h^{-1}(h(s)h(t))
=  s\ast_{_h}t,
\end{array}
\end{displaymath}
then we only need to verify the group axioms.

Firstly, we can see that (\ref{LCI}) allows an easy verification of the associative property
    \begin{displaymath}
\begin{array}{rcl}
\medskip
(t\ast_{_h}s)\ast_{_h}r & =& h^{-1}(h(t)h(s))\ast_{_h}r
= h^{-1}\left(h(h^{-1}(h(t)h(s)))h(r)\right)\\
\medskip
&=& h^{-1}(h(t)h(s)h(r)) \\
\medskip
&=& h^{-1}(h(t)h\{h^{-1}(h(s)h(r))\}) \quad \textnormal{since $h(s)h(r)=h(h^{-1}(h(s)h(r)))$}\\
\medskip
&=& t\ast_{_h}h^{-1}(h(s)h(r))
= t\ast_{_h}(s\ast_{h} r). 
\end{array}
\end{displaymath}

By using (\ref{LCI}) and $e_{*}$ from (\ref{inverse}) we have that
\begin{displaymath}
t\ast_{_h}e_{*} = h^{-1}(h(t)h(e_{*}))=h(h^{-1}(e_{*}))=t \quad \textnormal{for any $t\in \mathbb{R}$} 
\end{displaymath}
and it follows that $e_{*}$ is the unit element.

By using again (\ref{LCI}) and $t^{*-1}$ from (\ref{inverse}) we have that
\begin{displaymath}
t\ast_{_h}t^{*-1} = h^{-1}(h(t)h(t^{*-1}))= h^{-1}\left(\frac{h(t)}{h(t)}\right)= h^{-1}(1)=e_{*},
\end{displaymath}
and the result follows.
\end{proof}

In order to illustrate the previous abstract constructions, we now consider a concrete choice of
growth rate $h$. This example shows explicitly how the induced group operation $*$, the neutral
element $e_*$ and the inverse $t^{*-1}$ can be computed, and how the usual real line is
reparametrized around a new base point.

\begin{example}\label{ex 1}
    If $h(t)=e^{(t-2)^{3}}$ 
we deduce that $h^{-1}(t)=2+\sqrt[3]{\ln(t)}$. Then  $e_{*}=2$ and $t^{*-1}=4-t.$
\end{example} 

\begin{remark}
As we have said, the above result has been stated without proof in \cite{EPR} and \cite{JFP}
encompassing Proposition \ref{Groupe} since consider a general growth rate $h\colon J\to (0,+\infty)$, where $J=(a_{0},+\infty)$
and $a_{0}$ can be either a finite number or $-\infty$.
\end{remark}

A direct consequence from (\ref{LCI}) and (\ref{inverse}) is the pair of identities:
\begin{equation}
\label{emulation}
h(t\ast_{_h}s)= h(t)h(s) \quad \textnormal{and} \quad h(t^{\ast -1}) =  \frac{1}{h(t)},
\end{equation}
which also implies that
\begin{equation*}
(t\ast_{_h}s^{\ast-1})^{\ast -1}=s\ast_{_h}t^{\ast-1}.
\end{equation*}

\begin{remark}
As stated in \cite{EPR}, the following properties are a direct consequence of \eqref{LCI} combined
with the fact that $h(\cdot)$ and $h^{-1}(\cdot)$ are strictly increasing:
\begin{subequations}
  \begin{empheq}[left=\empheqlbrace]{align}
  &
t \leq s \quad \textnormal{if and only if} \quad u\ast t \leq u \ast s \quad \textnormal{for any $u\in \mathbb{R}$} \label{group4}, \\
&
t \leq s \quad \textnormal{if and only if} \quad t\ast u \leq s \ast u \quad \textnormal{for any $u\in \mathbb{R}$} \label{group4b}, \\
&
t \leq s \quad \textnormal{if and only if} \quad  s^{\ast-1}\leq t^{\ast-1} \label{group4c}, 
\end{empheq}
\end{subequations}
where $\leq$ denotes the classical order in $\mathbb{R}$. See also \eqref{group4c}.

A direct consequence from \eqref{group4} and \eqref{group4b} is that $(\mathbb{R},*,\leq)$ is a totally ordered group, we refer to \cite[p.7]{DNR}.
\end{remark}

Furthermore, the order $\leq$ is also characterized by the \textnormal{positive cone} $[e_{*},+\infty)$ in the sense that
$s\leq t$ if and only if $e_{*}\leq t *_{h} s^{*-1}$, and we refer to Lemma 2.1 from \cite{EPR} for details. When considering
the positive cone as a subgroup $([e_{*},+\infty))$, a well known topic of ordered abelian groups, see for example \cite[p.2]{Chiswell}, is the construction of the \textbf{absolute value} $|\cdot|_{\ast}\colon (\mathbb{R}, *_{_h}) \to ([e_{\ast},+\infty),*_{h})$:
\begin{equation}
\label{abs}
|t|_{\ast}=\left\{\begin{array}{ccr}
t &\textnormal{if}& e_{*}\leq t, \\
t^{\ast -1} &\textnormal{if}& \,\, t<e_{\ast}.
\end{array}\right.
\end{equation}

A byproduct of the above defined absolute value is given by the following results, whose
prove is identical to the standard absolute value.

\begin{lemma}{\label{desigTrian}}
The triangle inequality is satisfied:
\begin{equation}
\label{DT}
|t\ast_{_h}s|_{\ast} \leq |t|_{\ast} \ast_{_h}|s|_{\ast} \quad \textnormal{for any $t$ and $s$ in $(\mathbb{R},*)$}.
\end{equation}
\end{lemma}

\begin{lemma}{\label{metrica}}
The function $d \colon (\mathbb{R},\ast_{_h}) \times (\mathbb{R},\ast_{_h}) \to   ([e_{\ast},+\infty),\ast_{_h})$ defined by
\begin{equation}
\label{distance}
d(t,s):=|t\ast_{_h}s^{\ast -1}|_{\ast}
\end{equation}
verifies the following properties:
\begin{itemize}
    \item[d1)] $d(t,s)=e_{\ast}$ if and only if $t=s$,
    \item[d2)] $d(t,s)=d(s,t)$,
    \item[d3)] $d(t,s)\leq d(t,u)\ast_{_h}d(u,s)$.
\end{itemize}    
\end{lemma}

\begin{lemma} 
\label{L4}
Given $L>e_{\ast}$ it follows that
\begin{equation*}
|u|_{\ast}\leq L  \iff L^{\ast-1}\leq u \leq L.
\end{equation*}
\end{lemma}

\subsection{A vector space induced by the growth rate $h$}
Notice that the results devoted to the ordered group $(\mathbb{R},*)$ were deduced
by using the first composition law (\ref{LCI}) and its consequences. Moreover, it is well known that
(see \emph{e.g.}, \cite[p.11]{Bourbaki}) any abelian group is a $\mathbb{Z}$--module by considering
the external composition law
\begin{equation*}
\begin{array}{rcl}
\mathbb{Z} \times \mathbb{R} & \to &  \mathbb{R}  \\
(k,t) &\mapsto &  t^{*k}:=\left\{\begin{array}{rcl}
 \underbrace{t \ast \cdots \ast t}_{k-\textnormal{times}} &\textnormal{if}& k>0 \\
 e_{\ast}  &\textnormal{if}&  k=0 \\
 \underbrace{t^{\ast -1} \ast \cdots \ast t^{\ast -1}}_{k-\textnormal{times}} &\textnormal{if}& k<0.
 \end{array}\right.
\end{array} 
\end{equation*}

By using recursively the identities (\ref{emulation}) combined with the bijectivity of
$h$ we can see that
$$
t^{*k}=h^{-1}(h(t^{*k}))=h^{-1}(h(t)^{k}) \quad \textnormal{for any $k\in \mathbb{Z}$},
$$
and the above external composition law can be revisited as
\begin{equation}
\label{LCE-bis2}
\begin{array}{rcl}
\mathbb{Z} \times \mathbb{R} & \to &  \mathbb{R}  \\
(k,t) &\mapsto &  t^{*k}= k\odot t:=h^{-1}\left(h(t)^{k}\right),
\end{array} 
\end{equation}

A big novelty of this article is to notice that (\ref{LCE-bis2}) can be extended to (\ref{LCE}), that is,
the $\mathbb{Z}$--module $(\mathbb{R},*_{h})$ is an $\mathbb{R}$--vector space:
\begin{lemma}
\((\mathbb R,*_{_h}, \odot)\) is a vector space over the field \((\mathbb R,+,\cdot)\).
\end{lemma}

\begin{proof}
By Proposition \ref{Groupe}, we know that $(\mathbb{R},\ast_{_h})$ is an abelian group. 
In consequence, we only need to verify the scalar multiplication axioms. In order to do
that, let $\alpha,\beta\in\mathbb R$ and $s,t\in (\mathbb R, *_{_h},\odot)$.


\medskip

\noindent i) \textit{Distributivity over scalar addition}. 
It is,   $
(\alpha+\beta)\odot t=(\alpha\odot t)*_{_h}(\beta\odot t).$

In fact, by using (\ref{LCE}) we can easily verify that
  \[
  h\big((\alpha+\beta)\odot t\big)=h(t)^{\alpha+\beta}=h(t)^{\alpha}h(t)^{\beta}
  =h\big(\alpha\odot t\big)\,h\big(\beta\odot t\big).
  \]

Now, by using (\ref{emulation}) we can see that
   \[
  h\big((\alpha+\beta)\odot t\big)=h\big((\alpha\odot t)*_{_h}(\beta\odot t)\big),
  \]
and the property follows since $h$ is invertible.
  \medskip
  
\noindent ii) \textit{Distributivity over vector operation}. We will verify that
$$
\alpha\odot(t*_{_h}s)=(\alpha\odot t)*_{_h}(\alpha\odot s).
$$

In fact, by using (\ref{emulation}) and (\ref{LCE}) we have 
  \[
  h\big(\alpha\odot(t*_{_h}s)\big)=h(t*_{_h}s)^{\alpha}
  =h(t)^{\alpha}h(s)^{\alpha}=h\big(\alpha\odot t\big)\,h\big(\alpha\odot s\big)
  =h\big((\alpha\odot t)*_{_h}(\alpha\odot s)\big),
  \]
and the property follows.

 \medskip 
\noindent iii) \textit{Compatibility with field multiplication}. The identity $
 (\alpha\beta)\odot t=\alpha\odot(\beta\odot t)
$
is a direct consequence of (\ref{LCE}) since
  \[
  h\big((\alpha\beta)\odot t\big)=h(t)^{\alpha\beta}=(h(t)^{\beta})^{\alpha}
  =h\big(\beta\odot t\big)^{\alpha}
  =h\big(\alpha\odot(\beta\odot t)\big).
  \]
 \end{proof}

\begin{proposition}\label{Prop_odot}
Let \(a, b \in (\mathbb R,*_{_h}, \odot)\)  and $\alpha \in (\mathbb R,+,\cdot).$  The following properties are verified
\begin{enumerate}
    \item The additive inverse of $a$ is $(-1)\odot a=a^{*-1}.$
    \medskip
    
  \item  $(\alpha \odot a)^{*-1} = \alpha \odot a^{*-1} = (-\alpha) \odot a.$

\medskip

\item If $a \le b$ and  $\alpha > 0$ then $\alpha \odot a \le \alpha \odot b.$ If $\alpha <  0$ then $\alpha \odot a \ge  \alpha \odot b.$

    \end{enumerate}

\end{proposition}

\begin{proof}
(1) Since $h((-1)\odot a)=h(a)^{-1}$ and hence $a*_{_h}((-1)\odot a)=e_*$.
\medskip

(2) Note that 
$$
(\alpha \odot a)^{*-1} = h^{-1}\left( \frac{1}{h(\alpha \odot a)}\right) = h^{-1}\left( \frac{1}{h(a)^\alpha }\right) = \alpha \odot a^{*-1} = (-\alpha) \odot a.
$$

\medskip
(3) Assume that $a \le b,$ it is, $e_* \le b *_{_h} a^{*-1} $. Then 
$$
\begin{array}{rl}
\medskip
(\alpha \odot b) *_{_h} (\alpha \odot a)^{*-1} & = (\alpha \odot b) *_{_h} (\alpha \odot a^{*-1}) = \alpha \odot (b *_{_h}  a^{*-1}) \\
\medskip
& = h^{-1}(h(b *_{_h}  a^{*-1})^{\alpha}).
\end{array}
$$
Since $e_* \le b *_{_h} a^{*-1}$ we have that $1 \le h(b *_{_h}  a^{*-1}).$ 

\medskip

If $\alpha > 0 $ then $1 \le h(b *_{_h}  a^{*-1})^\alpha$  and so $e_* \le h^{-1}(h(b *_{_h}  a^{*-1})^\alpha)$
concluding that $e_* \le (\alpha \odot b) *_{_h} (\alpha \odot a)^{*-1}$ equivalently $\alpha \odot a \le \alpha \odot b.$

\medskip

Now, if $\alpha < 0 $ then $1 \ge  h(b *_{_h}  a^{*-1})^\alpha$  and so $e_* \ge h^{-1}(h(b *_{_h}  a^{*-1})^\alpha)$
hence $ (\alpha \odot b) *_{_h} (\alpha \odot a)^{*-1} \le e_*$ next  $\alpha \odot b \le \alpha \odot a.$
\end{proof}

The Lemmas \ref{desigTrian},\ref{metrica} and \ref{L4} gathers results stated previously
in \cite{EPR} for the abelian group $(\mathbb{R},*_{h})$. We can strength these results
by considering the vector space $(\mathbb{R}, *_{_h}, \odot)$ over $(\mathbb{R}, +, \cdot)$ and 
the map $t\mapsto  |t|_* = d(t,e_*)$.

\begin{lemma}
$|\cdot|_*$ is a norm in  $(\mathbb R,*_{_h},\odot).$ 
\end{lemma}

\begin{proof}
Is clear that $|t|_* \ge e_*.$ From Lemma \ref{metrica} d1) we have that $|t|_*=e_* $ if and only if $t = e_*,$ and by Lemma \ref{desigTrian} that $|t*s|_* \le |t|_* * |s|_*.$ It only remains to prove that $|\alpha \odot t|_*= |\alpha|\, \odot \,  |t|_*,$ for all $t \in (\mathbb R,*,\odot)$ and $\alpha \in (\mathbb R,+,\cdot),$ here $|\cdot|$ is the usual absolute value  in $\mathbb{R}.$ By definition (\ref{abs}) we have that

\begin{equation*}
|\alpha \odot t|_{\ast}=\left\{\begin{array}{ccr}
\alpha \odot t &\textnormal{if}& e_{*}\leq \alpha \odot t, \\
(\alpha \odot t)^{\ast -1} &\textnormal{if}& \,\, \alpha \odot t <e_{\ast}.
\end{array}\right.
\end{equation*}
Since $h$ and $h^{-1}$ are increasing, we will consider four cases:

\medskip

\noindent \textit{Case i)}. If $t > e_*$ and $\alpha > 0$ then $h(t)^\alpha > 1$ consequently $\alpha \odot t= h^{-1}(h(t)^\alpha) > e_*,$ hence
    $$
    |\alpha \odot t|_* = \alpha \odot t = |\alpha| \odot |t|_*    $$

\noindent \textit{Case ii)}. If $t < e_*$ and $\alpha > 0$ we obtain $h(t)^\alpha < 1$ and $\alpha \odot t= h^{-1}(h(t)^\alpha) < e_*,$ hence
    $$
    |\alpha \odot t|_* = (\alpha \odot t)^{*-1} = h^{-1}\left(\frac{1}{h(\alpha \odot t)}\right) = 
h^{-1}\left(\frac{1}{h(t)^\alpha}\right)
   = |\alpha| \odot |t|_*    $$

\noindent \textit{Case iii)}. For $t > e_*$ and $\alpha < 0$ implies $h(t)^\alpha < 1$ and $\alpha \odot t= h^{-1}(h(t)^\alpha) < e_*,$ then
    $$
    |\alpha \odot t|_* = (\alpha \odot t)^{*-1} = h^{-1}(h(t)^{-\alpha}) = (-\alpha) \odot t   = |\alpha| \odot |t|_*    $$

\medskip

\noindent \textit{Case iv)}. If $t < e_*$ and $\alpha < 0$ then  $h(t)^\alpha > 1$ and $\alpha \odot t= h^{-1}(h(t)^\alpha) > e_*,$ then
    $$
    |\alpha \odot t|_*  = h^{-1}(h(t)^{\alpha}) = h^{-1}\left(\frac{1}{h(t)^{-\alpha}}\right) =  h^{-1}\left(\frac{1}{h(t)^{|\alpha|}}\right) = |\alpha| \odot t^{*-1}   = |\alpha| \odot |t|_*    .$$

\medskip
Moreover is clear that $|\alpha \odot e_*|_* = |e_*|_* = |\alpha| \odot e_* = e_* ,$ and conclude that 
$$
|\alpha \odot t|_* = |\alpha| \odot |t|_*
$$
for all $t \in (\mathbb R,*_{_h},\odot)$ and $ \ \alpha \in (\mathbb R,+,\cdot),$ obtaining the result desired.

\end{proof}

\subsection{The normed vector space $(\mathbb{R},*_{_h}, \odot, |\cdot|_*)$ is complete}

We will prove that the vector space $(\mathbb{R},*_{_h}, \odot)$ with the norm $|\cdot|_*$ is complete.
This requires to introduce definitions of convergence tailored to the norm $|\cdot|_{*}$ and the distance
(\ref{distance}). In the follows, we denote $\mathbb{R}_* = (\mathbb{R},*_{_h}, \odot, |\cdot|_*).$
\begin{definition}
The sequence $\{x_{n}\}_{n}\subset \mathbb{R}_*$ converges to $x$ when $n\to +\infty$ if
\begin{displaymath}
\forall \, \varepsilon >e_{\ast} \,\, \exists N:=N(\varepsilon)\in \mathbb{N} \quad\textnormal{such that} \quad
n>N \Rightarrow |x_{n}\ast x^{\ast-1}|_{\ast}<\varepsilon.
\end{displaymath}
\end{definition}

\begin{definition}
The sequence $\{x_{n}\}_{n}\subset \mathbb{R}_*$ is a Cauchy sequence if
\begin{displaymath}
\forall \, \varepsilon >e_{\ast} \,\, \exists N:=N(\varepsilon)\in \mathbb{N} \quad\textnormal{such that} \quad
n,m>N \Rightarrow |x_{n}\ast x_{m}^{\ast-1}|_{\ast}<\varepsilon.
\end{displaymath}
\end{definition}

We point out that if $h(t)=e^{t}$ we recover the classical definitions on convergent and Cauchy sequences.

\begin{definition}
We say that the sequence $\{x_{n}\}_{n}\subset \mathbb{R}_*$ is bounded if there exists $M>e_{\ast}$ such that
\begin{displaymath}
|x_{n}|_{\ast}\leq M  \quad \textnormal{for any $n\in \mathbb{N}$},
\end{displaymath}
or equivalently , from Lemma \ref{L4}, $M^{\ast-1} \leq x_{n} \leq M$ for any $n\in \mathbb{N}$.    
\end{definition}

\begin{lemma}
\label{CIB}
Any Cauchy sequence $\{x_{n}\}_{n}\subset \mathbb{R}_*$ is bounded.
\end{lemma}

\begin{proof}

Let $\{x_{n}\}_{n}$ be a Cauchy sequence, for  $L>e_{\ast}$ there exists $N:=N(L)$ such that
$ n,m>N $ implies $|x_{n}\ast_{_h}x_{m}^{\ast-1}|_{\ast}<L.$

\medskip

\noindent Now let $n>N$ and by using triangle inequality note that
\begin{displaymath}
\begin{array}{rcl}
\medskip
|x_{n}|_{\ast} & = & |x_{n}\ast_{_h} e_{\ast}|_* 
               = |x_{n}\ast_{_h} x_{N+1}^{\ast-1}\ast_{_h} x_{N+1}|_{\ast} \\
               \medskip
               & \leq & |x_{n}\ast_{_h} x_{N+1}^{\ast-1}|_{\ast} \ast_{_h} |x_{N+1}|_{\ast}  \\
               \medskip
               & \leq & L\ast_{_h} |x_{N+1}|_{\ast}.
\end{array}    
\end{displaymath}

Hence  
$
|x_{n}|_{\ast}  \leq  L \ast_{_h} |x_{N+1}|_{\ast}  \ \ \textnormal{for any $n>N$},
$
and we conclude that
$$
|x_{k}|_{\ast}\leq \max\left\{|x_{1}|_{\ast},\ldots,|x_{N}|_{\ast},|x_{N+1}|_{\ast},L \ast_{_h} |x_{N+1}|_{\ast}\right\} \quad\textnormal{for any $k\in \mathbb{N}$},
$$
and the boundedness follows.

\end{proof}

\begin{lemma}
\label{BW}
If the sequence $\{x_{n}\}_{n}\subset \mathbb{R}_*$ is bounded, there exists a convergent
subsequence $\{x_{n_{k}}\}_{k}$.
\end{lemma}

\begin{proof}
By hypothesis there exists $M>e_{\ast}$ such that 
$M^{\ast -1} \leq x_{n} \leq M$ for any $n\in \mathbb{N}$.  Then at least one of the subintervals of \textit{large} $M$,
either $I_{11}=[M^{\ast-1},e_{\ast})$ or $I_{12}=(e_{\ast},M]$ must contain infinite terms of the sequence.

\medskip

Let $M_{1}=\frac{1}{2}\odot M=h^{-1}(\sqrt{h(M)})>e_{\ast}$ and notice that $M_{1}^{\ast 2}=2\odot M_{1}=M.$ Now, consider the intervals of \textit{large} $M_{1}:$
\begin{itemize}
\item $I_{21}=[M^{\ast-1},M^{\ast-1}\ast_{_h} M_{1})$, 
\item $I_{22}=[M^{\ast-1}\ast_{_h} M_{1},M^{\ast-1}\ast_{_h} \underbrace{M_{1}\ast_{_h}M_{1}}_{=M_{1}^{*2}=M})=[M^{\ast-1}\ast_{_h} M_{1},e_{\ast})$,
\item $I_{23}=(e_{\ast},M_{1}^{\ast}]$,
\item $I_{24}=[M_{1},M_{1}\ast_{_h} M_{1}]=[M_{1}^{\ast},M]$,
\end{itemize}
and  at least one of these intervals must contain infinite terms of the sequence $\{x_{n}\}_{n}$. 

We can follow in a recursive way and to deduce that for any $n\in \mathbb{N}$, there will be an interval
of large $\frac{1}{2^{n}}\odot M$ containing infinite terms of the sequence $\{x_{n}\}_{n}$.

\medskip
Notice that  $\frac{1}{2^{n}}\odot M>e_{\ast}$ and the continuity of $h^{-1}(\cdot)$ implies that
\begin{equation}
\nonumber 
\lim\limits_{n\to +\infty}\frac{1}{2^{n}}\odot M = 
\lim\limits_{n\to +\infty}
h^{-1}\left(h(M)^{\frac{1}{2^{n}}}\right)=h^{-1}(1)=e_{\ast},
\end{equation}
which implies that
\begin{equation}
\label{ae1}
\nonumber\forall \delta > e_* \quad \exists N(\delta)\in \mathbb{N} \quad \textnormal{such that $k>N \Rightarrow \frac{1}{2^{k}}\odot M <\delta$}.
\end{equation}

Finally, given the sequence $\{x_{n}\}_{n}$ and the above $N(\delta)$, we construct a subsequence $\{x_{n_{k}}\}_{k}$ where the terms
are inside an interval $[c,c\ast M_{N}]$ of large $M_{N}=\frac{1}{2^{N}}\odot M$, after some finite time, that is 
\begin{equation*}
|x_{n_{k}}\ast_{_h} c^{\ast-1}|_{\ast}<\frac{1}{2^{N}}\odot M \quad \textnormal{for any $k > N$}.
\end{equation*}

Then, by considering the above constructed subsequence $\{x_{n_{k}}\}_{k}$ and using 
(\ref{ae1}), 
we deduce that
\begin{displaymath}
\begin{array}{rcl}
\displaystyle |x_{n_{k}}\ast_{_h} c^{\ast-1}|_{\ast}<  \frac{1}{2^{{N}}}\odot M  < \delta\\\\
\end{array}
\end{displaymath}
and, $\{x_{n_{k}}\}_{k}$ is convergent.

\end{proof}

\begin{lemma}
The vector space $\mathbb{R}_*$ is a Banach space.
\end{lemma}

\begin{proof}
We will prove that if $\{x_{n}\}_{n}\subset \mathbb{R}_*$ is a Cauchy sequence, then it is convergent.
To do that notice that by Lemma \ref{CIB} we know that $\{x_{n}\}_{n}$ is a bounded sequence. Then by Lemma   
\ref{BW},  there exists a convergent
subsequence $\{x_{n_{k}}\}_{k}$ such that
\begin{equation*}
\forall \, \varepsilon >e_{\ast} \,\, \exists N_{1}:=N_{1}(\varepsilon)\in \mathbb{N} \quad\textnormal{such that} \quad
k>N_{1} \Rightarrow |x_{n_{k}}\ast_{_h} x^{\ast-1}|_{\ast}<\varepsilon.
\end{equation*}

In addition, as $\{x_{n}\}_{n}$ is a Cauchy sequence, we also know that
\begin{equation*}
\forall \, \varepsilon >e_{\ast} \,\, \exists N_{2}:=N_{2}(\varepsilon)\in \mathbb{N} \quad\textnormal{such that} \quad
n,n_{k}>N_{2} \Rightarrow |x_{n}\ast_{_h} x_{n_{k}}^{\ast-1}|_{\ast}<\varepsilon.
\end{equation*}

Now, by using triangle's inequality (\ref{DT}) we can deduce that
\begin{displaymath}
\begin{array}{l}
|x_{n}\ast_{_h} x^{\ast-1}|_{\ast} = |x_{n}\ast_{_h} x_{n_{k}}^{\ast-1}\ast_{_h}x_{n_{k}}\ast_{_h} x^{\ast-1}|_{\ast} 
\leq  |x_{n}\ast_{_h}x_{n_{k}}^{\ast-1}|_{\ast}\ast_{_h}|x_{n_{k}}\ast_{_h}x^{\ast-1}|_{\ast}, 
\end{array}    
\end{displaymath}
which allow us to conclude that
\begin{displaymath}
\forall \, \varepsilon \ast_{_h} \varepsilon >e_{\ast} \,\, \exists N>\max\{N_{1},N_{2}\}\in \mathbb{N} \ \textnormal{such that} \ 
n>N \Rightarrow |x_{n}\ast_{_h} x^{\ast-1}|_{\ast}<\varepsilon \ast_{_h} \varepsilon
\end{displaymath}
and the convergence follows.

\end{proof}

\section{A Generalized Evolution Semigroup}

In this section, we will consider the normed vector space $ \mathbb{R}_* := (\mathbb{R}, *_{_h}, \odot, |\cdot|_*)$ and
the Banach space $X = (X, ||\cdot||_X)$ such that $\mathcal{B}(X)$ denotes the set of bounded linear operators in $X.$ Furthermore, $C(\mathbb{R}_*, X)$ is the space of all continuous functions $X-$ valued whereas $C_0(\mathbb{R}_*, X)$ denotes the Banach space of all functions in 
$C(\mathbb{R}_*, X)$ vanishing at $\pm \infty,$ endowed with the norm
$$
||u||_{*\infty} := \sup_{s\in \mathbb{R}_*} ||u(s)||_X.
$$

\begin{Remark}
  The space $ C_c(\mathbb{R}_*, X) := \{u \in C(\mathbb{R}_*, X) \, : \, \mathrm{supp}(u) \ \mbox{is compact} \}$ is dense in $C_0(\mathbb{R}_*, X).$ In fact, when considering $(\mathbb R_*,*_{_h},\odot)$ as a \emph{one–dimensional normed real vector space} with norm $|\cdot |_{*}$ together with a compact set
$B_r:=\{t\in\mathbb R_*:\ |t|_*\le r\}$ we can choose continuous cut-offs $\psi_n:[0,\infty)\to[0,1]$ with $\psi_n\equiv1$ in $[0,n]$ and $\psi_n\equiv0$ on $[n+1,\infty)$, and set $\eta_n(t):=\psi_n(|t|_*)\in C_c(\mathbb R_*,[0,1])$ (indeed $\mathrm{supp}\,\eta_n\subset B_{n+1}$). 
For $f\in C_0(\mathbb R_*,X)$ define $f_n:=\eta_n f\in C_c(\mathbb R_*,X)$. If $t\in B_n$ then $f_n(t)=f(t)$; if $t\notin B_n$ then $\|f(t)\|_X$ is small by the definition of $C_0(\mathbb R_*,X)$, hence $\|f-f_n\|_{*\infty}\to0$.
Thus, $C_c(\mathbb R_*,X)$ is dense in $C_0(\mathbb R_*,X)$.
\end{Remark}

We will modify the classical concept of an evolution semigroup associated with an evolution family on the half-line to fit the case $t \geq e_*$ and the linear flows may not agree with the restricted hypothesis of uniform exponential bounded growth (\ref{decroissance}). For this purpose we establish the following definitions.
\medskip

\begin{definition}
\label{demigroupe}
Let $(Y,||\cdot||_{Y})$ be a Banach space.  A family $\{T_t\}_{t \ge e_*}$ of bounded linear operators from $Y$ into $Y$ is an $h-$semigroup if 
 \begin{itemize}
     \item[(i)] $T_{e_*} =  \mathrm{Id}$
\item[(ii)] $T_{t*_{_h}s} = T(t)T(s)$ for every $t, s \ge e_*.$
\end{itemize}
 \end{definition}

In addition, we will say that  $\{T_t\}_{t \ge e_*}$  is a \textit{strongly continuous $h-$semigroup} if
$$
\lim_{t\to e_*^+} T_t y = y \quad \mbox{for every} \ \ y \in Y,
$$
or equivalently
\begin{displaymath}
\forall \, \varepsilon > 0 \,\, \exists \, \delta > e_* \ \textnormal{such that for any} \ y \in Y, \ 
e_* < t < \delta \Rightarrow \left\|T(t) y - y \right\|_{Y}<\varepsilon.
\end{displaymath}

\begin{definition}
    
 An $h-$evolution family on $X$ is a collection $\{U(t,s)\}_{t\ge s}$ of bounded linear operators acting on $X$  such that the following properties hold:

\begin{itemize}
    \item $U(t,t) = \mathrm{Id}, \ t \in \mathbb{R}_*;$
    \item $U(t,\tau) U(\tau, s)= U(t,s), \ t \ge \tau \ge s;$
    \item for each $x \in X$, the mapping $(t,s) \to U(t,s)x$ is continuous on 
    $$\{(t,s) \in \mathbb{R}_* \times \mathbb{R}_{*} \, : \, t \ge s\}.$$
    \end{itemize}
\end{definition}

We will say that the above $h-$evolution family $\{U(t,s)\}_{t\ge s}$ is $h-$bounded if there exist constants $\alpha > 0$ and $K \ge 1$ such that 
\begin{equation}{\label{hcotaU}}
    ||U(t,s)|| \le K [h(t*_{_h}s^{*-1})]^\alpha, \ \ \mbox{for} \ \ t \ge s.
\end{equation}

\begin{remark}
Notice that, by using the identities \eqref{emulation}, we can easily deduce that \eqref{hcotaU}
is equivalent to 
\begin{displaymath}
||U(t,s)|| \le K \left(\frac{h(t)}{h(s)}\right)^\alpha, \ \ \mbox{for} \ \ t \ge s.
\end{displaymath}
\end{remark}

For any
 $u \in C_0(\mathbb{R}_*, X)$ and $t \ge e_*$ we will define 
\begin{equation}\label{SemigPaper}
T_tu(s) := U(s, s*_{_h}t^{*-1}) u(s*_{_h}t^{*-1}) \quad \textnormal{for any $t\in [e_{*},+\infty)$}.
\end{equation}

The next result states that $s\mapsto T_{t}u(s) \in C_{0}(\mathbb{R}_{*},X)$ for any
$u\in  C_{0}(\mathbb{R}_{*},X)$ and $t\geq e_{*}$.

\begin{lemma}\label{C0semigroup}
  Let $\{U(t,s)\}_{t\ge s }$ be an $h-$bounded evolution operator with constants $K$ and $\alpha$. Then $\{T_t\}_{t \ge e_*}$ given by \eqref{SemigPaper} 
  is a strongly continuous $h-$semigroup over  $C_0(\mathbb{R}_*, X).$
  \end{lemma}

\begin{proof}
 Let $t \ge e_*$ and $u \in C_0(\mathbb{R}_*, X)$ fixed. Note that
 $$
 \begin{array}{rl}
 \medskip
 ||T_tu(s)||_X &   \le ||U(s, s*_{_h}t^{*-1})|| \,  ||u(s*_{_h}t^{*-1})||_X \\
 \medskip
      & \le K [h(s*_{_h}(t*_{_h}s^{*-1})^{*-1})]^\alpha \, ||u||_{*\infty} \\
      \medskip
& = K \displaystyle \, \left[ \frac{h(s)}{h(s*_{_h}t^{*-1})}\right]^\alpha \, ||u||_{*\infty}  = K \displaystyle \, \left[h(t)\right]^\alpha \, ||u||_{*\infty},
      \end{array}
 $$ 
where the last identities are a consequence of (\ref{emulation}). 

The function  $s\mapsto T_{t}u(s)$ belongs to $C(\mathbb{R}_{*},X)$ since, as stated by equation (\ref{SemigPaper}),
is a composition of functions which are continuous with respect to $s$. In addition, the first above inequality  combined with $u(\cdot)\in C_{0}(\mathbb{R}_*, X)$ implies that $T_tu(s) \to 0$ when $s \to \pm \infty$ and consequently the map $T_t: C_0(\mathbb{R}_*, X) \to C_0(\mathbb{R}_*, X) $ is well defined.  

Now, we will verify the $h-$semigroup properties stated in Definition \ref{demigroupe}. The property (i) follows
directly from $e_{*}^{*-1}=e_{*}$. Indeed, notice that:
\begin{displaymath}
T_{e_*} u(s) = U(s,s*_{_h}e_*) u(s*_{_h}e_*) = U(s,s) u(s) = u(s).
\end{displaymath}

Now, given any pair $t, \tau \ge e_*$ and noticing that   $s \ge s*_{_h}t^{*-1} \ge s*_{_h}t^{*-1}*_{_h} \tau^{*-1},$ we have: 
$$\begin{array}{rl} 
\medskip T_{(t*_{_h}\tau)} u(s) & = U(s,s*_{_h}(t*_{_h}\tau)^{*-1}) u(s*_{_h}(t*_{_h}\tau)^{*-1})\\
\medskip 
& = U(s,s*_{_h}t^{*-1}) U(s*_{_h}t^{*-1},s*_{_h}t^{*-1}*_{_h}\tau^{*-1}) u(s*_{_h} t^{*-1}*_{_h}\tau^{*-1})\\
\medskip
& =  U(s,s*_{_h}t^{*-1}) T_\tau u(s*_{_h}t^{*-1}) = T_t T_\tau u(s),
\end{array} 
$$   
hence $T_{e_*} = \mathrm{Id}$ and $T_{t*_{_h}\tau} = T_tT_\tau.$ 

\medskip

We need to prove that $\{T_t\}_{t\ge e_*}$ is strongly continuous. Since $C_c(\mathbb{R}_*, X)$ is dense in $C_0(\mathbb{R}_*, X)$  it is enough to prove that
$$
\lim_{n \to \infty} T_{t_n}u(s_n) - u(s_n) = 0 
$$
for $t_n \to e_*, s_n \in \mathbb{R}_*$ for any $n$ and $u \in C_c(\mathbb{R}_*, X)$ fixed. 

\medskip

Case 1. If $\{s_n\}$ is unbounded, that is, $s_n \to \pm \infty$ then $s_n*_{_h} t_n^{*-1} \to \pm \infty.$ Since $u$ is of support compact, we obtain that  $T_{t_n}u(s_n) - u(s_n) = 0$ for $n$ large enough.

\medskip

Case 2. Suppose that $\{s_n\}$ is bounded and let $\varepsilon > 0.$ Since $u \in C_c(\mathbb{R}_*, X)$ and taking a subsequence $\{s_{n_k}\}_k$ of  $\{s_n\}$   if necessary, we can assume that $s_n \to s_0$ and then $u(s_n) \to u(s_0)$ for $n \to \infty.$ From this, 
\begin{equation}\label{cont_u}
||u(s_n) - u(s_0)||_X < \eta, \quad \mbox{where} \quad \eta  = \frac{\varepsilon}{2(K + 1)}
\end{equation}
for $n$ large enough and $K$ from \eqref{hcotaU}. Using again that $u$ is for compact support, it follows that $u$ is uniformly continuous, there exists $\delta_1(\varepsilon) > e_*$ such that 
\begin{equation}\label{UnifCont_u}
||u(s') - u(s'')||_X < \eta \ \ \mbox{if} \ \ d(s',s'') < \delta_1.
\end{equation}

\medskip

On the other hand, for $s \ge \tau,$  we know that $(s,\tau) \to U(s,\tau)x$ is continuous for each $x \in X.$ Choosing $x = u(s_0),$ if $x = 0$ then $T_{t_n}u(s_n) - u(s_n) = 0$ for $n \to \infty,$ thus, we can assume that $u(s_0) \ne 0.$ Now,  there exists $\delta_2(\varepsilon, u(s_0)) > e_*$ such that $d(s_n,\tau) < \delta_2$ implies
\begin{equation}\label{Cont_U}
||U(s_n, \tau) u(s_0) - u(s_0)||_X < \eta.
\end{equation}

Put $\delta(\varepsilon) = \min\{ \delta_1(\varepsilon), \delta_2(\varepsilon, u(s_0))\}  > e_*.$

\medskip

\noindent Note that  $d(s_n*_{_h}t_n^{*-1}, s_n) \to e_*,$ hence exists $N = N(\varepsilon) \in \mathbb{N}$ such that $d(s_n*_{_h}t_n^{*-1}, s_n) < \delta(\varepsilon)$ for $n > N$,
from (\ref{UnifCont_u})  we have

\begin{equation}\label{cota_u_n}
||u(s_n*_{_h}t_n^{*-1}) - u(s_n)||_X < \eta.
\end{equation}

Finally, from (\ref{cont_u}), (\ref{Cont_U}) and (\ref{cota_u_n}) we obtain
$$
\begin{array}{rl}
\medskip
    ||T_{t_n}u(s_n) - u(s_n)||_X = & ||U(s_n, s_n*t_n^{*-1}) u(s_n*_{_h} t_n^{*-1}) - u(s_n)||_X \\
    \medskip
    \le & \  ||U(s_n, s_n*_{_h}t_n^{*-1}) (u(s_n*_{_h}t_n^{*-1}) - u(s_n))||_X \\
    \medskip
     &  + \, ||U(s_n, s_n*_{_h} t_n^{*-1})( u(s_n) - u(s_0))||_X  \\  
\medskip
     &  + \, ||U(s_n, s_n*_{_h}t_n^{*-1}) u(s_0) - u(s_0)||_X   +  || u(s_0) - u(s_n)||_X  \\  
\medskip
\le & 2K [h(s_n*_{_h}(s_n*_{_h}t_n^{*-1})^{*-1}]^\alpha \eta + 2 \eta \\
\medskip
= & 2 K h(t_n)^\alpha \eta + 2 \eta = 2 \eta ( K [h(t_n)]^\alpha +   1) .  \end{array}
$$
Therefore $||T_{t_n}u(s_n) - u(s_n)||_X < \varepsilon$ for sufficiently large $n$, which completes the proof.
\end{proof}

\subsection{Generator and Resolvent of an $h-$semigroup.}

\begin{definition}\label{generadorA}
Let $Y$ be a Banach space. The  infinitesimal generator $A:D(A)\subset Y \to Y$ of a $C_0-$semigroup $\{S_t\}_{t \ge 0}$ on $Y$  is the operator 
$$
Ax := \lim_{t \to 0^+} \frac{S_t x - x}{t}
$$
defined for every $y$ on its domain 
$$
D(A) := \{y \in Y \, : \, \lim_{t \to 0^+} \frac{S_t y - y}{t}  \ \ \mbox{exists}\}.
$$
\end{definition}

In our particular case, for $C_0(\mathbb{R}_*, X)$ with being $X$ Banach space, we define generator as follows.

\begin{definition}\label{generadorB}
We say that $B_h:D(B_h) \subset C_0(\mathbb{R}_*, X) \to C_0(\mathbb{R}_*, X)$ is the  generator of an $h-$semigroup strongly continuous $\{T_t\}_{t \ge e_*}$ if
$$
B_h w = \lim_{t \to e_ *^+} \frac{T_t w -  w}{\ln(h(t))}
$$
with  domain 
$
D(B_h) = \displaystyle  \{w \in C_0(\mathbb{R}_*, X) \, : \, \lim_{t \to e_*^+} \frac{T_t w - w}{\ln(h(t))}  \ \ \mbox{exists}\}.
$
\end{definition}

\begin{Remark}  It is important to note that
\begin{enumerate}
    \item The above definition, the concept of limit tells us that
\begin{displaymath}
\forall \, \varepsilon > 0 \,\, \exists \, \delta > e_* \ \textnormal{such that for} \ w \in D(B_h), \ 
e_* < t < \delta \Rightarrow \left\|\frac{T_t w - w}{\ln(h(t))} - B_hw \right\|_{*\infty}<\varepsilon.
\end{displaymath}

\item If $h(t) = e^t$ then $\mathbb{R}_* = (\mathbb{R}, +, \cdot, |\cdot|)$ and Definition \ref{generadorB} matches with Definition \ref{generadorA} considering $Y = C_0(\mathbb{R}, X).$
\end{enumerate}
\end{Remark}

\subsection{Auxiliar semigroup}

Let $\mu: \mathbb{R}_* \to \mathbb{R}$ be defined by $\mu(t) = \ln(h(t)).$ The function $\mu$ is a homeomorphism whose inverse $\mu^{-1}: \mathbb{R} \to \mathbb{R}_*$ is $\mu^{-1}(t) = h^{-1}(e^t).$ Note that $\mu$ is strictly increasing and verify
\begin{equation}\label{liminf}
    \lim_{s \to -\infty} \mu(s) = -\infty \ \ \mbox{and} \ \  \lim_{s \to +\infty} \mu(s) = +\infty.
\end{equation}

For the $h-$evolution family $U = \{U(t_0,s_0)\}_{t_0\ge s_0}$ and for $t,s \in \mathbb{R}$ we define
\begin{equation}\label{DefV}
V(t,s) = U(\mu^{-1}(t), \mu^{-1} (s)), \ \ \mbox{for} \ \ t \ge s,
\end{equation}
and we write $V = \{V(t,s)\}_{t \ge s}.$

\begin{lemma}\label{VexpAcot}
Let $\{U(t_0,s_0)\}_{t_0\ge s_0}$ an $h-$evolution family on $X$ $h-$bounded. Then 
$V = \{V(t,s)\}_{t\ge s}$ is an evolution family on $X$ satisfying 
$$
||V(t,s)|| \le K e^{\alpha(t-s)},   \ \ \mbox{for} \ \ t \ge s
$$
with $\alpha, K$ as in \eqref{hcotaU}.
\end{lemma}

\begin{proof}
Is clear that $V(t,s),$ for $t \ge s$, be a family of bounded linear operators acting over $X$, and verify:
\begin{itemize}
    \item $V(t,t) = U(\mu^{-1}(t), \mu^{-1} (t)) = \mathrm{Id}, \, t \in \mathbb{R}.$
\medskip

\item $V(t,\tau) V(\tau,s) = U(\mu^{-1}(t), \mu^{-1} (\tau)) U(\mu^{-1}(\tau), \mu^{-1} (s)), \, s \le \tau \le t$ in $\mathbb{R}.$ Since $\mu^{-1}$ is increasing we have $\mu^{-1}(s) \le\mu^{-1}(\tau) \le \mu^{-1}(t),$ then 

$$V(t,\tau) V(\tau,s) = U(\mu^{-1}(t), \mu^{-1} (s)) = V(t,s).$$
    \end{itemize}

    Moreover, from the definition of $\mu^{-1}$ and using that $U$ is $h-$bounded,  we obtain
    $$
    \begin{array}{rl}
    \medskip
    ||V(t,s)||  = &  ||U(\mu^{-1}(t), \mu^{-1} (s))|| \\
    \medskip
     \le  & K [h(\mu^{-1}(t)* (\mu^{-1}(s)))^{*-1}]^\alpha \\
     \medskip
     = & K \displaystyle \left[ \frac{h(\mu^{-1}(t))}{h(\mu^{-1}(s))}\right]^\alpha = K \, e^{\alpha(t -s)}
    \end{array}
    $$
\noindent for  $t \ge s,$ and thus $V$ is exponentially bounded in the sense given by (\ref{decroissance}).   
\end{proof}

From the classical theory of evolution semigroups, we are now able to define the evolution semigroup strongly continuous $\{S_t\}_{t \ge 0}$ on $C_0(\mathbb{R}, X)$ associated to the evolution family $V$:
$$
S_tv(s) = V(s,s-t) v(s-t), \ t \ge 0, \ v \in  C_0(\mathbb{R}, X), \ s \in \mathbb{R},
$$
we denote $(A,D(A))$ its generator, that is,  $A: D(A) \subset C_0(\mathbb{R}, X) \to C_0(\mathbb{R}, X)$ with 

\begin{equation*}
D(A) = \{v \in C_0(\mathbb{R}, X)\, : \, \lim_{t \to 0^+} \frac{S_t v - v}{t}  \ \ \mbox{exists}\}.
\end{equation*}

\noindent 

Since $\mu$ is bijective, for $t, s \in \mathbb{R}$ there are $s_0, t_0 \in \mathbb{R}_*$ such that $\mu(s_0) = s$ and $\mu(t_0) = t$. Note that, $0 \le t = \mu(t_0) = \ln(h(t_0))$, 
from which it follows that $ h(t_0) \ge 1$ and therefore $t_0 \ge e_*.$ Moreover, by using the definition of $\mu^{-1}$ combined with (\ref{emulation}) we can deduce that
\begin{equation}\label{mu*}
\begin{array}{rl}
\medskip
    \mu^{-1}(\mu(s_0) - \mu(t_0)) = & \mu^{-1}(\ln(h(s_0)) - \ln(h(t_0))) \\
    \medskip
    = & \displaystyle \mu^{-1}\left(\ln\left(\frac{h(s_0)}{h(t_0)}\right)\right) \\
    \medskip
    = & \displaystyle  h^{-1}\left(\frac{h(s_0)}{h(t_0)}\right) = s_0*t_0^{*-1}.
    \end{array}
\end{equation}

Hence, by (\ref{DefV}) we can see that
$$
\begin{array}{rl}
\medskip
S_tv(\mu(s_0)) = & V(\mu(s_0), \mu(s_0) - \mu(t_0)) \, v(\mu(s_0)-\mu(t_0)) \\
\medskip
= & U(s_0, \mu^{-1}(\mu(s_0) - \mu(t_0))) \, v(\mu(s_0)-\mu(t_0)).
\end{array}
$$
Letting $w = v \circ \mu,$ using (\ref{liminf}) is clear that $w \in C_0(\mathbb{R}_*, X),$ from (\ref{SemigPaper}) and (\ref{mu*}) we obtain

$$
\begin{array}{rl}
\medskip
S_t(w \circ \mu^{-1})(\mu(s_0)) = & U(s_0, \mu^{-1}(\mu(s_0) - \mu(t_0))) \, w(\mu^{-1}(\mu(s_0)-\mu(t_0))) \\
\medskip
= & U(s_0, s_0*_{_h}t_0^{*-1}) \, w(s_0*_{_h} t_0^{*-1}) = T_{t_0} w(s_0),
\end{array}
$$
where, $(S_t(w\circ \mu^{-1})) \circ \mu = T_{t_0}w$ for any $w \in C_0(\mathbb{R}_*, X).$

\medskip 

Now, define the operator $\mathcal{F}: C_0(\mathbb{R}_*, X)  \to C_0(\mathbb{R}, X)$ by $\mathcal{F}(w) = w \circ \mu^{-1}.$ Is clear that $\mathcal{F}$ is an invertible, bounded linear operator, with its inverse $\mathcal{F}^{-1}v = v \circ \mu.$ It follows that
$$
T_{t_0}w = (S_t(\mathcal{F}(w))) \circ \mu = \mathcal{F}^{-1} (S_t(\mathcal{F}(w))), \ \ \mbox{for every} \ \ w \in C_0(\mathbb{R}_*, X),$$

\noindent and thus 

\begin{equation}
\label{SimSem}
T_{t_0} = \mathcal{F}^{-1} S_t \mathcal{F},  \ \ \mbox{with} \quad  t \ge 0 \quad \textnormal{such that} \quad \mu^{-1}(t) = t_0 \ge e_*.    
\end{equation}

\medskip


The next result characterizes the $h$--generator $B_{h}$ of $\{T_{t_{0}}\}_{t_{0}\geq e_{*}}$ (cf. Definition \ref{generadorB}) in terms of the infinitesimal
generator of $\{S_{t}\}_{t\geq 0}$.

\begin{theorem}\label{AgenB}
  Let $\mu, \ \mathcal{F}$ as above, $\{T_{t_0}\}_{t_0 \ge e_*}$ the evolution $h-$semigroup associated with $U$ and $\{S_t\}_{t \ge 0}$ the evolution semigroup associated with $V$ and generator $(A,D(A))$.
  Then the $h-$generator $B_h$ of $\{T_{t_0}\}_{t_0 \ge e_*}$ is given by
\begin{equation}\label{GenerSem}
B_hw = (\mathcal{F}^{-1} A \mathcal{F}) w = A (w \circ \mu^{-1}) \circ \mu
\end{equation}
with domain 
$$
D(B_h) = \displaystyle  \{w \in C_0(\mathbb{R}_*, X) \, : \, w \circ \mu^{-1} \in D(A)\}.
$$

  \end{theorem}

\begin{proof} Let $w \in D(B_h) \subset C_0(\mathbb{R}_*, X).$ By Definition \ref{generadorB} we have that 
$$
\begin{array}{rcl}
\medskip 
\mathcal{F}B_hw  & = & \displaystyle\mathcal{F}  \lim_{t_0 \to e_*^+} \frac{T(t_0) w - w}{\ln(h(t_0))} \\
& = & \displaystyle \lim_{t_0 \to e_*^+} \frac{\mathcal{F}T(t_0) w -\mathcal{F} w}{\ln(h(t_0))}   \\
\medskip
& = & \displaystyle  \lim_{t \to 0^+} \frac{\mathcal{F} T(\mu^{-1}(t)) w -\mathcal{F} w}{t} \\
\medskip
& = & \displaystyle  \lim_{t \to 0^+} \frac{S(t) \mathcal{F} w - \mathcal{F} w}{t} = A \mathcal{F} \omega\\
\end{array}
$$

\noindent for $ \mathcal{F}w \in D(A).$  Consequently, if $(w \circ \mu^{-1}) \in D(A)$ we obtain that $B_h  w  = \mathcal{F}^{-1} A \, \mathcal{F} w = A (w \circ \mu^{-1}) \circ \mu.$

\end{proof}

Let $A\colon D(A)\subset C_{0}(\mathbb{R},X)\to C_{0}(\mathbb{R},X)$ be the generator of the $C_0-$semigroup $\{S_t\}_{t \ge 0}$ defined on $C_0(\mathbb{R}, X).$ The resolvent set of $A$ is composed by the complex numbers $\lambda$ for which $\lambda I - A$ is invertible and is denoted by $\rho(A)$, i.e., 
$$
\rho(A) = \{ \lambda \in \mathbb{C} \, : \, (\lambda I - A)^{-1} \ \mbox{is a bounded linear operator in} \ C_0(\mathbb{R}, X) \}.
$$
The family $R(\lambda,A) = (\lambda I - A)^{-1}, \, \lambda \in \rho(A)$ of a bounded linear operator is called the resolvent of $A$ and $\sigma(A) = \mathbb{C}- \rho(A)$ is the spectral set of $A$.

\medskip

Similarly, for $B_h$  the linear operator with domain $D(B_h)$  generator of the $h-$semigroup $\{T_{t_0}\}_{t_0 \ge e_*}$ defined on $C_0(\mathbb{R}_*, X),$ the $h-$resolvent set $\rho(B_h)$ of $B_h$ is the set of all complex number $\lambda$ for which $\lambda I - B_h$ is invertible, i.e., 
$$
\rho(B_h) = \{ \lambda \in \mathbb{C} \, : \, (\lambda I - B_h)^{-1} \ \mbox{is a bounded linear operator in} \ C_0(\mathbb{R}_*, X) \}.
$$
The family $R(\lambda,B_h) = (\lambda I - B_h)^{-1}, \, \lambda \in \rho(B_h)$ of a bounded linear operator is called the $h-$resolvent of $B_h$ and $\sigma(B_h) = \mathbb{C}- \rho(B_h)$ the spectral set of $B_h$.

\begin{lemma}\label{espectros=}
In the conditions of Theorem \ref{AgenB} we have that $\sigma(B_h)=\sigma(A).$
\end{lemma}

\begin{proof}
From the Theorem  \ref{AgenB} it is enough to observe that 
$$
\lambda \in \rho(A) \iff (\lambda I - A)^{-1} = \mathcal{F} (\lambda I - B_h)^{-1} \mathcal{F}^{-1} \iff \lambda \in \rho(B_h).$$
\end{proof}

\begin{Remark}
Consequently, spectral properties of the \(h\)-evolution semigroup \(\{T_{t_0}\}_{t_0\ge e_*}\) follow directly from the classical evolution semigroup theory via the conjugates \(T_{t_0}= \mathcal{F}^{-1}S_t \mathcal{F}\) and \(B_h= \mathcal{F}^{-1}A\mathcal{F}\), with $t_0 = \mu^{-1}(t).$ See figure below.

\begin{center}
\begin{tikzcd}[column sep=large, row sep=large]
C_0(\mathbb{R}_*, X) 
  \arrow[r, "\mathcal{F}"] 
& 
C_0(\mathbb{R}, X) 
  \arrow[d, "{\{S_t\}}_{t \ge 0}"] 
\\
C_0(\mathbb{R}_*, X) 
  \arrow[u, "{\{T_{t_0}\}}_{t_0 \ge e_*}"] 
& 
C_0(\mathbb{R}, X) 
  \arrow[l, "\mathcal{F}^{-1}"]
\end{tikzcd}
\end{center}

\end{Remark}

\subsection{Measure on the $\mathbb{R}_*$}

Following the ideas developed in \cite{JFP}, we define a measure $\mu_*$ on $\mathbb{R}_*$, which is invariant under the action of elements of $\mathbb{R}_*.$ For this purpose we will assume that $h$ is a derivable function. The measure is an absolutely continuous and  its Radon-Nikodym derivative is the logarithmic derivative of $h$, that is, for Borel measurable set $R \subset \mathbb{R}_*$ we define
\begin{equation*}
 \mu_*(R) = \int_R \frac{h'(\tau)}{h(\tau)} dm(\tau), 
\end{equation*}
where $m(\cdot)$ is the Lebesgue measure. To see that this an invariant measure, it is enough to prove its is invariant for compact intervals, as these sets generate the Borel $\sigma-$algebra.

\begin{lemma}
Given a compact interval $[a,b] \subset \mathbb{R}_*$ the measure $\mu_*$ is invariant under operation $*$  by elements of $\mathbb{R}_*,$ that is, for any $\gamma \in \mathbb{R}_*$ and $\alpha \in \mathbb{R}$ 
$$
\mu_*([a,b]) = \mu_*([\gamma*a,\gamma*b]).$$
Moreover, for $\alpha \in \mathbb{R}$
$$
 \mu_*(\alpha \odot [a, b]) = |\alpha| \, \mu_*([a, b]).
 $$
\end{lemma}

\begin{proof}
The measure of the interval $[a,b] \subset \mathbb{R}_*$ is
$$
\mu_*([a, b]) = \int_a^b \frac{h'(\tau)}{h(\tau)} dm(\tau) = \ln(h(b)) - \ln(h(a)) = \mu(b) - \mu(a).
$$  
Given any constant $\gamma \in \mathbb{R}_*$ we obtain
$$
\mu_*([\gamma *_{_h} a,\gamma *_{_h} b]) = \ln\left( \frac{h(\gamma*_{_h}a)}{h(\gamma*_{_h}b)}\right) = \ln\left( \frac{h(a)}{h(b)}\right) = \mu_*([a, b]).$$

Note that $\mu_*([a, b]) = -\mu_*([b, a]).$ If $a\le b$ and  $\alpha < 0$ from Proposition \ref{Prop_odot} we have that  $\alpha \odot b < \alpha  \odot a$
$$
\begin{array}{rl}
\medskip
\mu_*(\alpha \odot [a, b]) & = \mu_*([\alpha \odot b, \alpha \odot a]) \\
\medskip
& = \ln(h(\alpha \odot a)) - \ln(h(\alpha \odot b)) \\
\medskip
&= \ln(h(a)^\alpha) - \ln(h(b)^\alpha) \\
\medskip
&= \alpha (\ln(h(a)) - \ln(h(b))) \\
\medskip
&= -\alpha \mu_*([a,b]) = |\alpha| \, \mu_*([a,b]). 
\end{array}
$$   

For $a\le b$ and  $\alpha > 0$ from Proposition \ref{Prop_odot} we have that  $\alpha \odot a < \alpha  \odot b$
$$
\begin{array}{rl}
\medskip
\mu_*(\alpha \odot [a, b]) & = \mu_*([\alpha \odot a, \alpha \odot b]) \\
\medskip
&= \ln(h(b)^\alpha) - \ln(h(a)^\alpha) \\
\medskip
&= \alpha (\ln(h(b)) - \ln(h(a))) \\
\medskip
&=  |\alpha| \, \mu_*([a,b]).
\end{array}
$$

\end{proof}

The following Lemma proved in \cite{JFP} establishes the existence of a partition of $\mathbb{R}_*$ into intervals of
constant $\mu_*$ measure. 

\begin{lemma}
For any $\gamma \in (e_*, \infty)$ the intervals $I_k= [k \odot \gamma, (k+1) \odot \gamma]$ with $k \in \mathbb{Z}$ define a partition of $\mathbb{R}_*$ into sets of constant $\mu_*$ measure.
\end{lemma}

A consequence of the above result is:

\begin{corollary}
 The measure $\mu_*$  is $\sigma-$finite.  
\end{corollary}

\section{Applications of evolution $h-$semigroup to general $h-$dichotomies}

In this section we adopt a semigroup viewpoint for $h$-evolution families. We construct
the evolution $h$-semigroup $\{T_t\}_{t\ge e_*}$ on $C_0(\mathbb{R}_*,X)$ and relate its
infinitesimal generator to the asymptotic behaviour encoded by the $h$-dichotomy.
In particular, we obtain a spectral characterization of $h$-dichotomy in terms of the
invertibility of the generator and the absence of spectrum on the imaginary axis, thus
extending to the $h$-framework the classical results known in the exponential case.

\begin{definition}
 Let $\{T_{t_0}\}_{t_0 \ge e_*}$ an $h-$semigroup strongly continuous. We will say that $\{T_{t_0}\}_{t_0 \ge e_*}$ is $h-$hyperbolic if there exists a projection $\mathcal{P}_h$ on the Banach space $(Y,||\cdot||_{Y})$ satisfying 
 $$
 T_{t_0} \mathcal{P}_h = \mathcal{P}_hT_{t_0}, \ \ t_0 \ge e_{*}, 
 $$
 and the following conditions hold:
\begin{itemize}
    \item the map $T_{t_0}|_{\mathcal{Q}_hY}: \mathcal{Q}_hY \to \mathcal{Q}_hY$ is invertible for each $t \ge e_*$, where $\mathcal{Q}_h = I - \mathcal{P}_h$;

    \item there exist $\nu > 0$ and $N \ge 1$ such that 
    $$
    ||T_{t_0} \mathcal{P}_hy||_Y \le N \, h(t_0)^{-\nu} ||y||_{Y} \ \ \mbox{and} \ \ ||(T_{t_0})^{-1}_{\mathcal{Q}_h} \,  \mathcal{Q}_hy||_Y \le N \, h(t_0)^{-\nu} ||y||_{Y}    $$
    \noindent for $t_0 \ge e_*.$
\end{itemize}

\end{definition}

In the case where $h(t) = e^t$ we recover 
the classical concept of hyperbolicity: 

\begin{definition}
\label{Def10}
 Let $(Z,||\cdot||_{Z})$ be a Banach space and $\{S_{t}\}_{t \ge 0}$ a semigroup on $Z$ strongly continuous. We will say that $\{S_{t}\}_{t \ge 0}$ is hyperbolic if there exists a projection $\mathcal{P}$ on $Z$ that satisfies 
 \begin{equation}
 \label{commut}
 S_t \mathcal{P} = \mathcal{P} S_{t}, \ \ t \ge 0, 
 \end{equation}
 and the following conditions hold:
\begin{itemize}
    \item the map $S_{t}|_{\mathcal{Q}Z}: \mathcal{Q} Z \to \mathcal{Q}Z$ is invertible for each $t \ge 0$, where $\mathcal{Q} = I - \mathcal{P}$;

    \item thereexist $\nu > 0$ and $N \ge 1$ such that 
    $$
    ||S_{t} \mathcal{P}z||_Z \le N \, e^{-\nu t} ||z||_{Z} \ \ \mbox{and} \ \ ||(S_{t}|_{\mathcal{Q}Z})^{-1}  \mathcal{Q}_hz||_Z\le N \, e^{-\nu t} ||z||_{Z}    $$
    \noindent for $t \ge 0.$
\end{itemize}

\end{definition}

\medskip

In the following, we suppose that $U = \{U(t_0, s_0)\}_{t_0 \ge s_0}$ is an $h-$evolution family on $X$ which is $h-$bounded is the sense of (\ref{hcotaU}). Let   $\{T_{t_0} \}_{t_0 \ge e_*}$ be the corresponding evolution $h-$semigroup given by Proposition \ref{C0semigroup} and $(B_h, D(B_h))$ its generator. 

\medskip

We consider the evolution family $V$ given by (\ref{DefV}) and its corresponding evolution semigroup $\{S_t\}_{t \ge 0}$ on $C_0(\mathbb{R}, X)$, that is, 
$S_tv(s) = V(s,s-t) v(s-t), \ t \ge 0, \ v \in  C_0(\mathbb{R}, X), \ s \in \mathbb{R},$ and $(A,D(A))$ its generator.

\medskip

It is well-known that a strongly continuous semigroup is hyperbolic if and only if  $\sigma(S_t) \cap \mathbb{T} = \varnothing$ for some/all $t \ge 0$ and $\mathbb{T} = \{\lambda  \in \mathbb{C} \, : \, |\lambda| = 1\}.$ The structural projection $\mathcal{P}$ is the Riesz projection corresponding to the  operator $S_t$, which is given by 
$$
\mathcal{P} = \frac{1}{2\pi i} \int_{\mathbb{T}} (\lambda I - S_t)^{-1} d\lambda 
$$
for some fixed $t > 0,$ where $I: C_0(\mathbb{R},X) \to C_0(\mathbb{R},X)$ is 
the identity application.

\begin{lemma}\label{SG-hyperb}
Let $U$ be an $h-$evolution family which is $h$--bounded. The  semigroup $\{S_t\}_{t \ge 0}$  is hyperbolic if and  only if  the $h-$semigroup  $\{T_{t_0}\}_{t_0 \ge e_*}$ is $h-$hyperbolic.
\end{lemma}
\begin{proof}
 We assume that $\{S_t\}_{t \ge 0}$ is hyperbolic. Let $\mathcal{P}$ the Riesz's projection for $\{S_t\}_{t \ge 0}$ on $C_0(\mathbb{R},X),$ then by using (\ref{SimSem}) we can deduce that
$$
\begin{array}{rl}
\medskip
\mathcal{P} = & \displaystyle \frac{1}{2\pi i} \int_{\mathbb{T}} (\lambda I - S_t)^{-1} d\lambda \\
\medskip
= &\displaystyle \frac{1}{2\pi i} \int_{\mathbb{T}} (\lambda \mathcal{F}\mathcal{F}^{-1} - \mathcal{F} T_{t_0} \mathcal{F}^{-1})^{-1} d\lambda \\
\medskip
= & \displaystyle \frac{1}{2\pi i} \int_{\mathbb{T}} \mathcal{F} (\lambda I_* -  T_{t_0}) ^{-1} \mathcal{F}^{-1} d\lambda  
\medskip 
\end{array}
$$
where $I_*: C_0(\mathbb{R}_*,X) \to C_0(\mathbb{R}_*,X)$ is the identity application and $\mu^{-1}(t)  = t_0 \ge e_*.$ Hence, we define
$$
\mathcal{P}_h =\mathcal{F}^{-1} \mathcal{P} \mathcal{F}.
$$
We claim that $\mathcal{P}_h$ is a projection on $ C_0(\mathbb{R}_*,X).$ In fact, for $\mu^{-1}(t) = t_0 \ge e_*,$
and by using (\ref{SimSem}), (\ref{commut}) combined with the definition of $\mathcal{P}_{h}$ we can see that
$$
\begin{array}{rcl}
\mathcal{P}_h T_{t_0} & = &\mathcal{F}^{-1} \mathcal{P} \mathcal{F} \ \mathcal{F}^{-1} S_t \mathcal{F} \\
&=&\mathcal{F}^{-1}  \mathcal{P}  S_t \mathcal{F}\\
&=& \mathcal{F}^{-1} S_t \mathcal{P} \mathcal{F} \\
&=& \mathcal{F}^{-1} S_t  \mathcal{F} \ \mathcal{F}^{-1} \mathcal{P} \mathcal{F} \\
&=& T_{t_0} \mathcal{P}_h.
\end{array}$$
Now, write $\mathcal{Q}_h = I_* - \mathcal{P}_h$ and $(T_{t_0})_{\mathcal{Q}_h}:\mathcal{Q}_h C_0(\mathbb{R}_*,X) \to\mathcal{Q}_h C_0(\mathbb{R}_*,X)$ is invertible with $(T_{t_0})^{-1}_{\mathcal{Q}_h} (\mathcal{Q}_h w)= \mathcal{F}^{-1} S_t^{-1}(\mathcal{Q} F w),\mathcal{Q} = I - \mathcal{P}$ and $\mathcal{F}\mathcal{Q}_h = \mathcal{Q} \mathcal{F}.$

\medskip 
\noindent Remembering that $t = \mu(t_0) = \ln(h(t_0)), \, t_0 \ge e_*, $ from the hypothesis (cf. Definition \ref{Def10}) we have that 
$$
\begin{array}{rl}
\medskip
||T_{t_0}\mathcal{P}_h w||_{*\infty} &  \displaystyle = ||\mathcal{F}^{-1} S_t \mathcal{P} \mathcal{F} w||_{*\infty} \\
\medskip
& \le ||\mathcal{F}^{-1}|| \, || S_t \mathcal{P} \mathcal{F} w||_\infty \\
\medskip
& \le N e^{-\nu t}||\mathcal{F}^{-1}|| \,  ||\mathcal{F} w||_\infty \\
\medskip
& \le N e^{-\nu t} \,  ||w||_{*\infty} \\
\medskip 
& =  N e^{-\nu \mu(t_0)} \,  ||w||_{*\infty}  = N h(t_0)^{-\nu} \,  ||w||_{*\infty},
\end{array}$$
and 
$$
\begin{array}{rl}
\medskip
||(T_{t_0})^{-1}_{\mathcal{Q}_h} \, \mathcal{Q}_h w||_{*\infty} &  \displaystyle = ||\mathcal{F}^{-1} S_t^{-1} \mathcal{Q} \mathcal{F}w||_{*\infty} \\
\medskip
& \le ||\mathcal{F}^{-1}|| \, || (S_t)_{\mathcal{Q}}^{-1} \mathcal{Q} \mathcal{F}w||_\infty \\
\medskip
& \le N e^{-\nu t}||\mathcal{F}^{-1}|| \,  ||\mathcal{F}w||_\infty \\
\medskip 
& \le  N e^{-\nu \mu(t_0)} \,  ||w||_{*\infty}  = N h(t_0)^{-\nu} \,  ||w||_{*\infty}.
\end{array}$$

\medskip 
Reciprocally, we suppose that $\{T_{t_0}\}_{t_0 \ge e_*}$ is $h-$hyperbolic and let $\mathcal{P}_h$ its projection. We define $\mathcal{P} = \mathcal{F}  \mathcal{P}_h \mathcal{F}^{-1}$ and write  
$\mathcal{Q} = I - \mathcal{P}.$ Moreover, $S_t^{-1}(\mathcal{Q}v ) = \mathcal{F} T_{t_0} (\mathcal{Q}_h \mathcal{F}^{-1}v),$ for $v \in C_0(\mathbb{R},X).$ 
Similarly to the previous estimation, we conclude the result.

\end{proof}

In order to relate the spectra of $\{T_{t_0}\}_{t_0 \ge e_*}$ and $B_h$ to the hyperbolicity of the $h-$evolution family $\{U(t_0,s_0)\}_{t_0\ge s_0}$ we need some preliminary results. 
For the semigroup case, we have the following proposition, see \cite[Lemma 9.16]{EN}.

\begin{proposition}
 Let $\{S_t\}_{t\ge 0}$ be a hyperbolic evolution semigroup on $C_0(\mathbb{R},X)$ with corresponding projection  $\mathcal{P}$. Then $\varphi \mathcal{P} f = \mathcal{P}(\varphi f)$ for any $f \in C_0(\mathbb{R},X)$ where  $\varphi \in C_b(\mathbb{R})$, the space of all bounded continuous real-valued functions.
 \end{proposition}

\medskip

An important consequence of the above result is that 
$$
\mathcal{P} v(s) = P(s) v(s), \quad v \in C_0(\mathbb{R},X), \ s \in \mathbb{R},
$$
for some bounded, strongly continuous projection-valued function $\mathcal{P}: \mathbb{R}  \to \mathcal{B}(X)$ see \cite[Proposition 9.13]{EN}.

\medskip

The following definition \cite[Def. 1.1]{LRS} establishes the classical property of uniform exponential dichotomy of an evolution family.

\begin{definition}\label{DicExp}
The evolution family  $\{V(t,s)\}_{t \ge s }$ is said to admits an exponential dichotomy if 
\begin{itemize}

\item[(a)] there exist projections $P(t): X \to X, \, t \in \mathbb{R},$ and write $Q(t) = I - P(t)$ with $P(t) V(t,s) = V(t,s) P(s)$ and the restriction $V_Q(t,s): Q(s)X \to Q(t)X$ is invertible, for all $t \ge s$;

\item[(b)] there exist constants $\nu > 0, \,  N\ge 1$ such that 
$$
||V(t,s)P(s)|| \le N e^{-\nu (t - s)} \quad \mbox{and} \quad ||V_Q(t,s)^{-1}Q(t)|| \le N e^{-\nu (t -s)}, \ \ t \ge s.
$$
\end{itemize}
\medskip

\end{definition}

Notice that $||P(t)|| \le N$ for every $t \in \mathbb{R}.$ Furthermore, as in \cite[Lemma 4.2]{VRS}, one may prove that the mapping $t \to P(t)$ is strongly continuous and thus $P(\cdot) \in C_b(\mathbb{R}, \mathcal{B}_s(X)),$ the space of all bounded and continuous functions from $\mathbb{R}$ with values in $\mathcal{B}(X)$ endowed with the topology of strong convergence.

\medskip
 The hyperbolicity of the evolution semigroup characterizes the exponential dichotomy of the underlying evolution family, and thus the evolution  semigroups method provides a strong tool to study the exponential dichotomy of evolution families. More precisely, the result is stated in the following proposition (see, for instance, \cite[Theorem 3.17, Theorem 4.25]{CL} or \cite[Theorem VI.9.18]{EN}).

 \begin{proposition}\label{DicotV}
   Let $V$ be an exponentially bounded evolution family on a Banach space $X$, let $\{S_t\}_{t \ge 0 }$ be the associated evolution semigroup on  $C_0(\mathbb{R},X)$ and $A$ its generator. The following assertions are equivalent:

\begin{enumerate}
    \item $V$ admits an exponential dichotomy on $X$;

\item $\{ S_t\}_{t \ge 0 }$ is hyperbolic on $C_0(\mathbb{R},X)$;

\item  $\sigma(A) \cap i\mathbb{R} = \varnothing.$ In this case, $A$ is invertible and its inverse is given by 
$$
(A^{-1}f)(t) = - \int_{\mathbb{R}} \Gamma(t,s) f(s) ds, \quad \forall f \in C_0(\mathbb{R},X), \ \ t \in \mathbb{R},
$$
where 
$$
\Gamma(t,s) = \left\{ 
\begin{array}{cl}
\medskip
   V(t,s)P(s),  &  t > s, \\
    - V_Q(t,s)^{-1} Q(s) , & t < s.
\end{array}
\right.
$$

\end{enumerate}
   
   \end{proposition}

\medskip
 When dealing with $h-$evolution families, we will introduce the property of $h-$dichotomy as follows.

\begin{definition}\label{h-dich}
The $h-$evolution family  $\{U(t_0,s_0)\}_{t_0 \ge s_0 }$ is said to admits an  $h-$dichotomy if 
\begin{itemize}

\item[(a)] there exist projections $P_h(t_0): X \to X, \, t_0 \in \mathbb{R}_*,$ with 
$$P_h(t_0) U(t_0,s_0) = U(t_0,s_0) P_h(s_0),$$ 
write $Q_h(t_0) = I - P_h(t_0),$ and the restriction $U_{Q_h}(t_0,s_0): Q_h(s_0)X \to Q_h(t_0)X$ is invertible, for all $t_0 \ge s_0$;

\item[(b)] there exist constants $\nu > 0, \,  N\ge 1$ such that 
$$
||U(t_0,s_0)P_h(s_0)|| \le N \left(\frac{h(t_0)}{h(s_0)}\right)^{-\nu} \mbox{and} \ \ ||U_{Q_h}(t_0,s_0)^{-1}Q_h(t_0)|| \le N \left(\frac{h(t_0)}{h(s_0)}\right)^{-\nu},
$$
for $t_0 \ge s_0.$
\end{itemize}
\medskip

\end{definition}

We also have to $||P_h(t_0)|| \le N$ for every $t_0 \in \mathbb{R}_*.$ Following similar ideas to \cite[Lemma 4.2]{VRS}, one may prove that the mapping $t_0 \to P_h(t_0)$ is strongly continuous and thus $P(\cdot) \in C_b(\mathbb{R}_*, \mathcal{B}_{s}(X)),$ the space of all bounded and continuous functions from $\mathbb{R}_*$ with values in $\mathcal{B}(X)$ endowed with the topology of strong convergence:
$$
||P_h(t_0)||:= \sup\{ ||P_h(t_0)x||_X \, : \, ||x||_X \le 1\}.
$$

The main result of this article extends the equivalences from Proposition \ref{DicotV} to the framework
of $h$--evolution families.

\begin{theorem}\label{h-dichotomy}
   Let $U$ be an  $h-$evolution family which is $h-$bounded on a Banach space $X$, let $\{T_{t_0}\}_{t_0 \ge e_* }$ be the associated evolution $h-$semigroup on  $C_0(\mathbb{R}_*,X)$ and $B_h$ its generator. The following assertions are equivalent:

\medskip

\begin{enumerate}

\item $\{ T_{t_0}\}_{t_0 \ge e_* }$ is $h-$hyperbolic on $C_0(\mathbb{R}_*,X)$;

\item $U$ admits an $h-$dichotomy on $X$;

\item  $\sigma(B_h) \cap i\mathbb{R} = \varnothing.$ In this case, $B_h$ is invertible and, if $h$ is continuously differentiable then  its inverse is given by 
\begin{equation}\label{InversaB_h}
(B_h^{-1}g)(t_0) = - \int_{\mathbb{R}} \Gamma_h(t_0,s_0) g(s_0) d\mu_*, \quad \forall g \in C_0(\mathbb{R}_*,X), \ \ t_0 \in \mathbb{R}_*,
\end{equation}
where 
$$
\Gamma_h(t_0,s_0) = \left\{ 
\begin{array}{cl}
\medskip
   U(t_0,s_0)P_h(s_0),  &  t_0 > s_0, \\
    - U_{Q_h}(t_0,s_0)^{-1} Q_h(s_0) , & t_0 < s_0.
\end{array}
\right.
$$

\end{enumerate}
   
   \end{theorem}

\begin{proof}   Since $U$ is  an $h$--bounded $h-$evolution family on $X$, from Lemma \ref{VexpAcot} we have that $V$ given by (\ref{DefV}) is an evolution family exponentially bounded.

\medskip

 (1)$\Rightarrow (2).$  Due to  $\{ T_{t_0}\}_{t_0 \ge e_* }$ is $h-$hyperbolic on $C_0(\mathbb{R}_*,X)$ from Lemma \ref{SG-hyperb} we obtain that $\{S_t\}_{t \ge 0}$ is a semigroup hyperbolic on $C_0(\mathbb{R},X)$. Furthermore, by Proposition \ref{DicotV} we deduce that $V$ admits an exponential dichotomy on $X$.

   Let $P(t), \, t \in \mathbb{R},$ the  associated projection to the dichotomy of $V.$ We define 
   $$
   P_h(t_0) = (P \circ \mu)(t_0) = P(\mu(t_0)), 
   $$
and we claim that $P_h(t_0), \, t_0 \in \mathbb{R}_*$ is a projection associated to  $U$. In fact,  is clear that $P_h(t_0): X \to X$ is well defined, and
\medskip
$$
P_h(t_0) U(t_0, s_0) =  P(t) U(\mu^{-1}(t), \mu^{-1}(s)) = P(t) V(t,s) = V(t,s) P(s) =  U(t_0, s_0) P_h(s_0) $$
since $\mu(t_0) = t.$

\medskip

Note that  $Q_h(t_0) = I - P_h(t_0) = I - P(\mu(t_0)) = Q(\mu(t_0)),$ for $t_0 \ge s_0,$ we have
\medskip
$$U_{Q_h}(t_0,s_0)^{-1}Q_h(t_0) = U_{Q_h}(\mu^{-1}(t), \mu^{-1}(s))^{-1} Q(t) = V_Q(t,s)^{-1}Q(t), 
$$
due to $\mu(t_0) = t.$

Hence, considering that $\mu(t_0) = \ln(h(t_0)) = t$ combined with the property of dichotomy exponential of $V,$ there exist constants $\nu > 0, \,  N\ge 1$ such that
$$
||U(t_0,s_0)P_h(s_0)|| = ||V(t,s)P(s)|| \le N e^{-\nu (t-s)}
= N \left(\frac{h(t_0)}{h(s_0)}\right)^{-\nu} 
$$
and
$$ \ \ ||U_{Q_h}(t_0,s_0)^{-1}Q_h(t_0)|| = ||V_Q(t,s)^{-1}Q(t)|| \le N e^{-\nu (t -s)} =  N \left(\frac{h(t_0)}{h(s_0)}\right)^{-\nu},
$$
for $t_0 \ge s_0.$ So, we conclude that $U$ admits an $h-$dichotomy.

\medskip

(2) $\Rightarrow$ (3). Let $P_h(t_0)$ be the projection associated with $U.$ We write 
$$
P(t) = P_h(\mu^{-1}(t)).
$$
Similarly to the previous demonstration, we see that $P(t)$ is a projection that verify (a) and (b) in Definition \ref{DicExp}, hence $V$ admits an exponential dichotomy. From Proposition \ref{DicotV} we have $\sigma(A) \cap i\mathbb{R} = \varnothing.$  Using Lemma \ref{espectros=} we conclude that $\sigma(B_h) \cap i\mathbb{R} = \varnothing.$  It remains to prove the formula (\ref{InversaB_h}). For this pick $g \in C_{0}(\mathbb{R}_*, X)$ and set 
$$
w(t_0) = - \int_{\mathbb{R}_*} \Gamma_h(t_0,s_0) g(s_0) d\mu_*, \quad t_0 \in \mathbb{R}_*.
$$
It is clear that $g \circ \mu^{-1} \in C_{0}(\mathbb{R}, X).$ By Proposition \ref{DicotV} and, considering $\mu(t_0) = t \in \mathbb{R},$ $ \,  t_0 \in \mathbb{R}_*$ we get

$
\begin{array}{rl}
\medskip
A^{-1}(g \circ \mu^{-1})(\mu(t_0)) = A^{-1}(g \circ \mu^{-1})(t) & = \displaystyle- \int_{\mathbb{R}} \Gamma(t,s) (g \circ \mu^{-1})(s) ds\\
\end{array}
$
$$
\begin{array}{rl}
     \medskip
     & =  \displaystyle- \left[ \int_{-\infty}^t \Gamma(t,s)(g \circ \mu^{-1})(s) ds + \int_{t}^{+\infty} \Gamma(t,s)(g \circ \mu^{-1})(s) ds \right]\\
     \medskip
& =  \displaystyle- \left[ \int_{-\infty}^t V(t,s)P(s)(g \circ \mu^{-1})(s) ds - \int_{t}^{+\infty} V_Q(t,s)^{-1} Q(s) (g \circ \mu^{-1})(s) ds \right]\\
\medskip
& =  \displaystyle- \int_{-\infty}^t U(\mu^{-1}(t),\mu^{-1}(s))P(\mu(\mu^{-1}(s)) g(\mu^{-1}(s)) \, ds \\
\medskip
& \quad \displaystyle + \int_{t}^{+\infty} U_Q(\mu^{-1}(t),\mu^{-1}(s))^{-1} Q(\mu(\mu^{-1}(s)) g(\mu^{-1}(s)) ds \\
\medskip
& =  \displaystyle- \int_{-\infty}^{t_0} U(t_0,s_0)P(\mu(s_0)) g(s_0) \frac{h'(t_0)}{h(t_0)}\, dm(s_0) \\
\medskip
& \quad \displaystyle + \int_{t_0}^{+\infty} U_Q(t_0,s_0)^{-1} Q(\mu(s_0)) g(s_0)\frac{h'(t_0)}{h(t_0)}\, dm(s_0) \\
\medskip
& =  \displaystyle- \int_{-\infty}^{t_0} U(t_0,s_0)P_h(s_0) g(s_0) d\mu_*  
\displaystyle - \int_{t_0}^{+\infty} U_Q(t_0,s_0)^{-1} Q_h(s_0) g(s_0)\, d\mu_* \\
\medskip
& = \displaystyle - \int_{\mathbb{R}_*} \Gamma_h(t_0,s_0) g(s_0) d\mu_* = w(t_0).
     \end{array}
$$
Hence $A^{-1}(g \circ \mu^{-1}) \circ \mu = w$. Thus $w \circ \mu^{-1} \in D(A)$ and consequently $w \in D(B_h).$ From (\ref{GenerSem}) we have that $B_h^{-1}(g) = w$ which proves the desired formula.  

\medskip

(3) $\Rightarrow$ (1). Using the hypothesis we have $\sigma(A) \cap i\mathbb{R} = \varnothing,$ the Proposition \ref{DicotV} implies $\{ S_t\}_{t \ge 0 }$ is hyperbolic on $C_0(\mathbb{R},X)$ and from  Lemma \ref{SG-hyperb} we obtain that 
$\{ T_{t_0}\}_{t_0 \ge e_* }$ is $h-$hyperbolic on $C_0(\mathbb{R}_*,X).$

\end{proof}

We now return to the growth rate introduced in Example~\ref{ex 1}, namely
$h(t)=e^{(t-2)^3}$. In that earlier example we only computed the basic objects
induced by $h$, such as the neutral element $e_*$, the inverse $t^{*-1}$ and the
corresponding group operation $*$. In what follows: we construct the associated $h$-evolution family,
describe its $h$-dichotomy, construct the corresponding $h$-semigroup on 
$C_0(\mathbb{R}_*,X)$ and identify its infinitesimal generator and resolvent.
In this way, the cubic example $h(t)=e^{(t-2)^3}$ serves as a concrete model that 
ties together all the general results obtained in the manuscript.

\begin{example}
    Let $h:\mathbb R\to \mathbb{R}^+$ be given by $
h(t)=e^{(t-2)^3}.$ 
Then
\[
h^{-1}(\tau)=2+\sqrt[3]{\ln \tau},\qquad e_*:=h^{-1}(1)=2,
\]
and the inverse of $t$ in $(\mathbb R,*,\odot)$ is
\[
t^{*-1}=h^{-1}\!\Big(\frac{1}{h(t)}\Big)
=2+\sqrt[3]{-(t-2)^3}=4-t.
\]

Let $\{U(t_0,s_0)\}_{t_0\ge s_0}$ be an $h$–evolution family on $X$ satisfying the $h$–bound
\[
\|U(t_0,s_0)\|
\;\le\;K\Big(\frac{h(t_0)}{h(s_0)}\Big)^{\alpha}
=K\exp\!\big(\alpha\big[(t_0-2)^3-(s_0-2)^3\big]\big),
\qquad t_0\ge s_0.
\]

By Definition \ref{h-dich}, we say that $U$ admits an $h$–dichotomy, that is, there exist projections $P_h(t_0):X\to X$ and constants $N\ge1$, $\nu>0$ such that
\begin{align*}
P_h(t_0)U(t_0,s_0)&=U(t_0,s_0)P_h(s_0),\qquad t_0\ge s_0,\\[1mm]
\|U(t_0,s_0)P_h(s_0)\|
&\le N\exp\!\big(-\nu\big[(t_0-2)^3-(s_0-2)^3\big]\big),\\[1mm]
\|U_{Q_h}(t_0,s_0)^{-1}Q_h(t_0)\|
&\le N\exp\!\big(-\nu\big[(t_0-2)^3-(s_0-2)^3\big]\big),
\end{align*}
where $Q_h(t_0)=I-P_h(t_0)$ and $U_{Q_h}(t_0,s_0)$ is the restriction 
$U(t_0,s_0)\big|_{Q_h(s_0)X}:Q_h(s_0)X\to Q_h(t_0)X$.

  From Proposition \ref{C0semigroup}, the evolution $h$–semigroup $\{T_{t_0}\}_{t_0\ge e_*}$ on
$C_0(\mathbb R_*,X)$ is 
\[(T_{t_0}u)(s_0)
=U\Big(s_0,\;2+\sqrt[3]{(s_0-2)^3-(t_0-2)^3}\Big)\,
u\Big(2+\sqrt[3]{(s_0-2)^3-(t_0-2)^3}\Big),
\]
and its generator $B_h$ (Definition \ref{generadorB}) becomes
\[
B_h w
=\lim_{t_0\to 2^+}
\frac{T_{t_0}w-w}{\ln(h(t_0))}
=\lim_{t_0\to 2^+}
\frac{T_{t_0}w-w}{(t_0-2)^3},
\]
with domain
\[
D(B_h)=\Bigl\{w\in C_0(\mathbb R_*,X):
\lim_{t_0\to 2^+}\frac{T_{t_0}w-w}{(t_0-2)^3}\ \text{exists in }C_0(\mathbb R_*,X)\Bigr\}.
\]

Now, if we define the additive time variable $t:=\mu(t_0)=(t_0-2)^3$ and
\[
V(t,s):=U\big(\mu^{-1}(t),\mu^{-1}(s)\big)
=U\big(2+\sqrt[3]{t},\,2+\sqrt[3]{s}\big),\qquad t\ge s.
\]
Then $V$ is an exponentially bounded evolution family on $X$ and its associated
(classical) evolution semigroup $\{S_t\}_{t\ge0}$ on $C_0(\mathbb R,X)$ is
\[
(S_t v)(s)
=V(s,s-t)\,v(s-t).
\]

As shown in Theorem \ref{AgenB}, we have the conjugacy
\[
T_{t_0}
=\mathcal{F}^{-1}S_{\mu(t_0)} \mathcal{F},\qquad
B_h=\mathcal{F}^{-1} A \mathcal{F},
\]
where $A$ is the generator of $\{S_t\}_{t\ge0}$ and
$\mathcal{F}: C_0(\mathbb R_*,X)\to C_0(\mathbb R,X)$ is given by $(\mathcal{F}w)(s)=w(\mu^{-1}(s))$.

For  $h$, one has
\[
\frac{h'(t)}{h(t)}=\frac{3(t-2)^2 e^{(t-2)^3}}{e^{(t-2)^3}}
=3(t-2)^2,
\]
so the invariant measure on $\mathbb R_*$ is
$
d\mu_*(\tau)=3(\tau-2)^2\,d\tau.
$
Thus, the representation of $B_h^{-1}$ in Theorem~\ref{h-dichotomy} becomes
\[
(B_h^{-1}g)(t_0)
=-\int_{\mathbb R}\Gamma_h(t_0,s_0)\,g(s_0)\,3(s_0-2)^2\,ds_0,
\qquad t_0\in\mathbb R_*,
\]
where
\begin{displaymath}
\Gamma_h(t_0,s_0)=\left\{\begin{array}{rcl}
U(t_0,s_0)P_h(s_0), & & t_0>s_0,\\
-\,U_{Q_h}(t_0,s_0)^{-1}Q_h(s_0), & & t_0<s_0.
\end{array}\right.
\end{displaymath}

\end{example}

\section*{Data Availability Statement}
Data sharing not applicable to this article as no datasets were generated or analysed during the current study.

\section*{Conflict of interest}
Not applicable.


\begin{thebibliography}{99}

\bibitem{BT}
A.~A. Borichev and Y. Tomilov,
\textit{Optimal polynomial decay of functions and operator semigroups},
Math. Ann. \textbf{347} (2010), no.~2, 455--478.

\bibitem{Bourbaki} 
N. Bourbaki,
\textit{\'El\'ements de Math\'ematique. Alg\`ebre, Chapitre 2: Alg\`ebre Lin\'eaire},
Hermann, Paris, 1962.

\bibitem{CL}
C. Chicone and Y. Latushkin,
\textit{Evolution Semigroups in Dynamical Systems and Differential Equations},
Math. Surveys Monogr., Vol.~70, Amer. Math. Soc., Providence, RI, 1999.

\bibitem{Chiswell}
I. Chiswell,
\textit{Introduction to $\Lambda$-Trees},
World Scientific, Singapore, 2001.

\bibitem{DNR}
B. Deroin, A. Navas and C. Rivas,
\textit{Groups, Orders and Dynamics},
preprint, arXiv:1408.5805.

\bibitem{EPR}
H. Elorreaga, J.F. Pe\~na and G. Robledo,
\textit{Noncritical uniformity and expansiveness for uniform $h$-dichotomies},
Math. Ann. \textbf{393}  (2025) 1769–1795.


\bibitem{EN}
K. J. Engel and R. Nagel,
\textit{One-Parameter Semigroups for Linear Evolution Equations},
Grad. Texts in Math., Vol.~194, Springer, New York, 2000.

\bibitem{Henry-0}
D. Henry,
\textit{Geometric Theory of Semilinear Parabolic Equations},
Lecture Notes in Math., Vol.~840, Springer, Berlin, 1981.

\bibitem{Henry}
D.~B. Henry,
\textit{Exponential dichotomies, the shadowing lemma and homoclinic orbits in Banach spaces},
Resenhas IME-USP \textbf{1} (1994), no.~4, 381--401.

\bibitem{LRS}
Y. Latushkin, T. Randolph and R. Schnaubelt,
\textit{Exponential dichotomy and mild solutions of nonautonomous equations in Banach spaces},
J. Dyn. Differential Equations \textbf{10} (1998), no.~3, 489--510.

\bibitem{LM}
N. Lupa and M. Megan,
\textit{Exponential dichotomies of evolution operators in Banach spaces},
Monatsh. Math. \textbf{174} (2014), no.~2, 265--284.

\bibitem{Lupa2}
N. Lupa and L.~H. Popescu,
\textit{Generalized exponential behavior on the half-line via evolution semigroups},
Carpathian J. Math. \textbf{38} (2022), no.~3, 691--705.

\bibitem{Lupa}
N. Lupa and L.~H. Popescu,
\textit{Generalized evolution semigroups and general dichotomies},
Results Math. \textbf{78} (2023), no.~3, Paper No.~112, 26~pp.


\bibitem{Martin} R.H. Martin Jr,
\textit{Conditional stability and separation of solutions to differential equations}, 
J. Differ. Equ. \textbf{13} (1973) 81--105.


\bibitem{MSS}
M. Megan, B. Sasu and A.~L. Sasu,
\textit{On nonuniform exponential dichotomy of evolution operators in Banach spaces},
Integral Equations Operator Theory \textbf{44} (2002), no.~1, 71--78.

\bibitem{Muldowney}J.~S. Muldowney, Dichotomies and asymptotic behaviour for linear differential systems, Trans. Amer. Math. Soc. {\bf 283} (1984), no.~2, 465--484.

\bibitem{NaulinPinto} R.~M. Naulin and M. Pinto, Roughness of $(h,k)$-dichotomies, J. Differential Equations {\bf 118} (1995), no.~1, 20--35.

\bibitem{Pazy}
A. Pazy,
\textit{Semigroups of Linear Operators and Applications to Partial Differential Equations},
Appl. Math. Sci., Vol.~44, Springer, New York, 1983.

\bibitem{JFP}
J.~F. Pe\~na and S. Rivera--Villagr\'an,
\textit{On uniform asymptotic $h$-stability for linear nonautonomous systems},
Electron. J. Qual. Theory Differ. Equ. (2025), Article~19, 13~pp.

\bibitem{VQP}
V.~Q. Phong,
\textit{On the exponential stability and dichotomy of $C_{0}$-semigroups},
Studia Math. \textbf{132} (1999), no.~2, 141--149.

\bibitem{SS}
R.J. Sacker and G.R. Sell,
\textit{Dichotomies for linear evolutionary equations in Banach spaces},
J. Differ. Equ. \textbf{113} (1994), 17--67.

\bibitem{VRS}
N. Van Minh, F. R\"abiger and R. Schnaubelt,
\textit{Exponential stability, exponential expansiveness, and exponential dichotomy of evolution equations on the half-line},
Integral Equations Operator Theory \textbf{32} (1998), no.~3, 332--353.

\end{thebibliography}
\end{document}